\documentclass[11pt]{amsart}%

\usepackage{amssymb}

\usepackage{tikz}
\usetikzlibrary{calc}

\newtheorem{theorem}{Theorem}
\newtheorem{lemma}[theorem]{Lemma}

\newtheorem{prop}[theorem]{Proposition}

\newtheorem{ques}{Question}

\newtheorem{exam}{Example}

\newcommand{\R}{\mathbb{R}}
\newcommand{\N}{\mathbb{N}}

\newcommand{\K}{\mathcal{K}}
\newcommand{\M}{\mathcal{M}}
\newcommand{\B}{\mathcal{B}}

\newcommand{\ctm}{\mathfrak{c}}
\newcommand{\diag}{\Delta}
\newcommand{\tr}{\ge_T}
\newcommand{\cl}[1]{\overline{#1}}

\newcommand{\cof}[1]{\mathop{cof}{#1}}

\newcommand{\down}[1]{\, \downarrow \! \! #1\,}

\begin{document}

\title{$P$-Paracompact and $P$-Metrizable Spaces}

\author{Ziqin Feng}
\address{Department of Mathematics and Statistics, Auburn University, Auburn, AL~36849, USA.}
\email{zzf0006@auburn.edu}

\author{Paul Gartside} 
\address{Department of Mathematics, University of Pittsburgh, Pittsburgh, PA~15260, USA.} 
\email{gartside@math.pitt.edu} 
\thanks{This work was partially supported by a grant from the Simons Foundation  (\#209751 to Paul Gartside)}

\author{Jeremiah Morgan}
\address{Department of Mathematics, University of Pittsburgh, Pittsburgh, PA~15260, USA.}
\email{jsm63@pitt.edu}



\begin{abstract}
Let $P$ be a directed set and $X$ a space. A collection $\mathcal{C}$ of subsets of $X$ is \emph{$P$-locally finite}  if $\mathcal{C}=\bigcup \{ \mathcal{C}_p : p \in P\}$ where  (i) if $p \le p'$ then $\mathcal{C}_p \subseteq \mathcal{C}_{p'}$ and (ii) each $\mathcal{C}_p$ is locally finite. Then $X$ is \emph{$P$-paracompact} if every open cover has a $P$-locally finite open refinement. Further, $X$ is \emph{$P$-metrizable} if it has a $(P \times \N)$-locally finite base.  This work provides the first detailed study of $P$-paracompact and $P$-metrizable spaces,  particularly in the case when $P$ is a $\K(M)$, the set of all compact subsets of a separable metrizable space $M$ ordered by set inclusion.
\end{abstract}

\maketitle

\section{Introduction}

Let $P$ be a directed set and $X$ a space (all spaces are assumed to be Tychonoff, unless otherwise stated). A collection $\mathcal{C}$ of subsets of $X$ is \emph{$P$-point finite} (respectively, \emph{$P$-locally finite}) if $\mathcal{C}=\bigcup \{ \mathcal{C}_p : p \in P\}$ where  (i) if $p \le p'$ then $\mathcal{C}_p \subseteq \mathcal{C}_{p'}$ and (ii) each $\mathcal{C}_p$ is point finite (respectively, locally finite). Then $X$ is \emph{$P$-metacompact} (respectively, \emph{$P$-paracompact}) if every open cover has a $P$-point finite (respectively, $P$-locally finite) open refinement. Further, $X$ is \emph{$P$-metrizable} if it has a $(P \times \N)$-locally finite base. Note that  $1$-metacompact is metacompact, $1$-paracompact is paracompact, and $1$-metrizable is metrizable. This work provides the first detailed study of $P$-paracompact and $P$-metrizable spaces, with particular emphasis on the case when $P$ is a $\K(M)$, the set of all compact subsets of a separable metrizable space $M$ ordered by set inclusion.

These covering properties were introduced to characterize certain types of compact spaces arising in analysis. Indeed a compact space $X$ is (i) Eberlein, (ii) Talagrand, or (iii) Gul'ko  if and only if $X^2\setminus\Delta$ is (i)$'$ $\N$-metacompact, (ii)$'$ $\K(\omega^\omega)$-metacompact, or (iii)$'$ $\K(M)$-metacompact, for some separable metrizable $M$ (respectively). The equivalence of (i) and (i)$'$ is due to Gruenhage, \cite{Gru84}, while the other two equivalences are due to Garcia, Orihuela and Oncina, \cite{GOO04}.  The first two authors gave in \cite{FG} a systematic and uniform development of the theory of those compact spaces $X$ such that $X^2\setminus \Delta$ is $P$-metacompact (\emph{$P$-Eberlein compact}), giving alternative characterizations in terms of almost subbases, bases, networks and point networks. 

A key result from \cite{FG} is that  a $P$-Eberlein compact space is Corson compact (embeds in a $\Sigma$-product of lines) if $P$ has calibre $(\omega_1,\omega)$ (every uncountable subset of $P$ contains an infinite subset with an upper bound), but if $P$ is not calibre $(\omega_1,\omega)$ then \emph{every} compact space of weight no more than $\omega_1$ is $P$-Eberlein compact. This demonstrates the critical role that calibre~$(\omega_1,\omega)$ plays. However Gul'ko (and so Talagrand and Eberlein) compacta have pleasant properties that general Corson compacta do not (for example, ccc Gul'ko compacta are metrizable in ZFC). The proofs  use special properties of directed sets of the form $\K(M)$. The authors would like to know what --- if anything --- is special about directed sets of the form $\K(M)$ as compared to general directed sets with calibre $(\omega_1,\omega)$.

Here we show (Theorems~\ref{main_result} and~\ref{main_conv}) that a pseudocompact space $X$ such that $X^2\setminus\Delta$ is $P$-paracompact for some $P$ with calibre $(\omega_1,\omega)$ is metrizable. However, if $P$ is not calibre $(\omega_1,\omega)$ then there is a compact non-metrizable space $X$ such that $X^2\setminus\Delta$ is  $P$-paracompact. This gives an `optimal' extension of  Gruenhage's result, \cite{Gru84}, that a compact space $X$ with paracompact $X^2\setminus\Delta$ is metrizable. We also show that a separable space is Lindel\"{o}f if  $P$-paracompact, and metrizable if $P$-metrizable, for some $P$ with calibre $(\omega_1,\omega)$.

However developing further results about $P$-paracompact and $P$-metrizable spaces apparently needs $P$ to be a $\K(M)$. These results depend on alternative characterizations which do not (overtly) refer to a directed set. For example, we show (Theorem~\ref{P_metric}) that a space is $\K(M)$-metrizable for some separable metrizable $M$ if and only if it has a weakly $\sigma$-locally finite base (a base $\mathcal{B}=\bigcup_n \mathcal{B}_n$ where for each point $x$ we have $\mathcal{B}=\bigcup \{ \mathcal{B}_s : x \text{ is locally finite in } \mathcal{B}_s\}$). With these in place we show that a first countable space is paracompact if and only if it is $\K(M)$-paracompact, normal and countably paracompact (Theorem~\ref{char_pcpt}); $\K(M)$-paracompact normal Moore spaces are metrizable (Theorem~\ref{NM}); first countable $\K(M)$-paracompact ccc spaces are Lindel\"{o}f; and $\K(M)$-metrizable ccc spaces are metrizable (Theorem~\ref{kmccc}). We also construct  examples  distinguishing the relevant covering properties, and demonstrating the necessity of additional hypotheses. In particular, we are able to show that in some cases we cannot replace `$P=\K(M)$' with a general `$P$ with calibre $(\omega_1,\omega)$'.

\section{Preliminaries}

\subsection{Directed Sets}

A \emph{directed set} is a partially ordered set such that any two elements have an upper bound. If $p$ is an element of a directed set $P$ we write $\down{p}$ for the set $\{ p' \in P : p' \le p\}$. The \emph{cofinality}, $\cof{P}$, of a directed set $P$ is the minimal size of a cofinal set in $P$. A directed set is \emph{Dedekind complete} if every subset with an upper bound has a least upper bound, and is \emph{countably directed} if every countable subset has an upper bound. For any space $X$ the set $\K(X)$ of all compact subsets of $X$ (including the empty set) is a Dedekind complete directed set when ordered by inclusion. For a separable metrizable $M$, we have  $\cof{\K(M)} \le |\K(M)| \le \ctm$. Other examples of Dedekind complete directed sets include $[S]^{<\omega}$ and $[S]^{\le \omega}$ of all finite (respectively, countable) subsets of a set $S$, again ordered by set inclusion.

A directed set $P$ has \emph{calibre $(\omega_1,\omega)$} if each uncountable subset must contain an infinite subset with an upper bound.  
\begin{lemma}\label{when_calom1om}
For any separable metrizable space $M$, and any set $S$, $\K(M)$ and $[S]^{\le \omega}$ have calibre $(\omega_1,\omega)$, but $[S]^{<\omega}$ only has calibre $(\omega_1,\omega)$ if $S$ is countable.
\end{lemma}

There is a more general notion of \emph{relative calibre $(\omega_1,\omega)$} for subsets of a poset: if $P'$ is a subset of a poset $P$, then $P'$ has relative calibre $(\omega_1,\omega)$ in $P$ if each uncountable subset of $P'$ contains an infinite subset with an upper bound in $P$.

A directed set $Q$ is a \emph{Tukey quotient} of a directed set $P$, and we write $P \ge_T Q$, if there is a map $\phi:P\to Q$ taking cofinal sets of $P$ to cofinal sets in $Q$.
   More generally, we can define \emph{relative Tukey quotients} as follows.  A subset $C$ of $P$ is said to be \emph{cofinal for a subset} $P'$ of $P$ if for every $p'\in P'$, there is some $c\in C$ such that $p'\le c$.  Then, if $P'\subseteq P$ and $Q' \subseteq Q$, we write $(P',P)\ge_T (Q',Q)$ if there is a map $\phi: P\to Q$  taking sets cofinal for $P'$ in $P$ to sets cofinal for $Q'$ in $Q$. If $P'=P$, then we may abbreviate this to $P\ge_T (Q',Q)$. See \cite{GMa} for the proof of the next lemma. The non-relative version is well-known. 
   \begin{lemma}\label{Ded}
If $(P',P) \ge_T (Q',Q)$ and $Q$ is Dedekind complete then there is an order-preserving map $\phi: P \to Q$ such that $\phi (P')$ is cofinal for $Q'$ in $Q$.
\end{lemma}

If $P \ge_T Q$ and $Q\ge_T P$, then we say $P$ and $Q$ are \emph{Tukey equivalent} and write $P=_T Q$.  This is an equivalence relation. We note that every countable directed set without a maximum is Tukey equivalent to $\N$.   Tukey equivalence was introduced by Tukey \cite{Tukey} in order to study the cofinal behavior of directed sets. It has since been intensively studied both in purely order theoretic terms and in the context of directed sets arising naturally in topology and analysis. 

Write $\M$ for the set of all separable metrizable spaces, and $\K(\M)$ for the set of Tukey equivalence classes, $[\K(M)]$, of $M$ from $\M$. Then $\K(\M)$ is a directed set under Tukey quotients.

Tukey quotients preserve (relative) calibres.
\begin{lemma} \label{tukey_cal}
If $(P',P)\ge_T (Q',Q)$ and $P'$ has relative calibre $(\omega_1,\omega)$ in $P$, then $Q'$ also has the same relative calibre in $Q$.
\end{lemma}

Let us connect our two types of standard directed sets. A subset of $\R$ is \emph{totally imperfect} if every compact subset is countable. Bernstein sets are examples of totally imperfect subsets of the reals which have size $\ctm$.
\begin{lemma} \label{bernstein}
If $B\subseteq\R$ is a totally imperfect subset of $\R$ with size $\ctm$ and $S$ is any set with size $|S|\le\ctm$, then $\K(B) \ge_T [S]^{\le\omega}$.
\end{lemma}

\begin{proof} 
Since $|S|\le\ctm$ we also have $|[S]^{\le \omega}| \le \ctm$, so fix a surjection $f:B\to [S]^{\le \omega}$.  Define $\phi:\K(B)\to [S]^{\le\omega}$  by $\phi(K) =  \bigcup \{f(x) : x\in K\}$. This is a well-defined map into $[S]^{\le\omega}$ since each compact subset of $B$ is countable.  Clearly $\phi$ is order-preserving and surjective,  and hence a Tukey quotient map.
\end{proof}

The next two results are from a paper \cite{GMa} by the second author and Mamatelashvili. The first says when subsets of $\K(\M)$ have upper bounds, it implies that $\K(\M)$ is countably directed. The second asserts the existence of an antichain of maximal possible size in $\K(\M)$. 

\begin{theorem}[\cite{GMa}] \label{km_bounded}
Let $\{M_\alpha : \alpha < \kappa\}$ be a family of separable metrizable spaces.

(i) If $\kappa \le \ctm$ then there is a separable metrizable space $M$ such that for all $\alpha$ we have $\K(M) \ge_T \K(M_\alpha)$.

(ii) If $\kappa > \ctm$ and the $M_\alpha$'s are all distinct subsets of a given separable metrizable space (or pairwise non-homeomorphic) then for any separable metrizable space $M$ there is an $\alpha$ such that $\K(M) \not\ge_T (M_\alpha, \K(M_\alpha))$.
\end{theorem}

\begin{theorem}[\cite{GMa}]\label{bigantichain}
There is a $2^\ctm$-sized family $\mathcal{A}$ of separable metrizable spaces such that if $A$ and $A'$ are distinct elements of $\mathcal{A}$ then $\K(A) \not\ge_T (A',\K(A'))$ and $\K(A') \not\ge_T (A,\K(A))$.
\end{theorem}

The following result is from \cite{GM}, by the second two authors.

\begin{theorem}[\cite{GM}]\label{Sigma}
Let $\{M_\alpha : \alpha < \kappa\}$ be a family of separable metrizable spaces.
Then the $\Sigma$-product, $\Sigma \K(M_\alpha)$, taken with base point $\mathbf{0}=(\emptyset)_\alpha$ and considered with the product order, has calibre $(\omega_1,\omega)$. 
\end{theorem}

\subsection{Topology}

\subsubsection{Absolutely Closed Covers of Subspaces}

For any space $Z$ denote by $\mathop{CL}(Z)$ the set of all closed subsets of $Z$. For any set $S$ we write $\mathbb{P}(S)$ for the set of all subsets of $S$.

\begin{lemma}\label{lf_sum_sm}
Let $Y$ be a metrizable space that is locally separable. Let $A$ be any subspace of $Y$.

(i) There is a directed set $P$ with calibre $(\omega_1,\omega)$ such that $P \ge_T (A, CL(Y) \cap \mathbb{P}(A))$.  Moreover, $P = \Sigma \K(M_\alpha)$ where each $M_\alpha$ is separable metrizable.

(ii) If the weight of $Y$ is $\le \ctm$, then we can take $P=\K(M)$ in (i) where $M$ is some separable metrizable space.
\end{lemma}

\begin{proof} Using local separability, regularity and paracompactness of $Y$, we can find a closed locally finite cover $\mathcal{C} = \{C_\alpha : \alpha < w(Y)\}$ of $Y$ of separable sets. Define $A_\alpha = C_\alpha \cap A$.

For (i), define $P = \Sigma \K(A_\alpha)$, which has calibre $(\omega_1,\omega)$ by Lemma~\ref{Sigma}.  Any collection $\{K_\alpha : \alpha < w(Y)\}$ with $K_\alpha \in \K(A_\alpha)$ is locally finite in $Y$ since $\mathcal{C}$ is locally finite.  Therefore, $\bigcup_\alpha K_\alpha$ is closed in $Y$ and also a subset of $A$.  The map $\Sigma \K(A_\alpha) \to CL(Y) \cap \mathbb{P}(A)$ given by $(K_\alpha)_\alpha \mapsto \bigcup_\alpha K_\alpha$ is then order-preserving, and its image covers $A$.

For (ii), we use Theorem~\ref{km_bounded} to find a separable metrizable space $M$ and Tukey quotient maps $\phi_\alpha : \K(M) \to \K(A_\alpha)$ for each $\alpha < w(Y) \le \ctm$.  Then define $\phi : \K(M) \to CL(Y) \cap \mathbb{P}(A)$ by $\phi(K) = \bigcup_\alpha \phi_\alpha(K)$, which witnesses $\K(M) \ge_T (A, CL(Y) \cap \mathbb{P}(A))$.
\end{proof}

\begin{lemma}\label{YAforSigmaKA}
Let $\{A_\alpha : \alpha < \kappa\}$ be a family of separable metrizable spaces and define  $P = \Sigma \K(A_\alpha)$. Then there is a locally separable metrizable space $Y$ with a subspace $A$ homeomorphic to $\bigoplus \{ A_\alpha : \alpha < \kappa\}$ such that:

(i) $P \ge_T (A, CL(Y) \cap \mathbb{P}(A))$, and

(ii) if $Q$ is a directed set such that $Q \ge_T (A, CL(Y) \cap \mathbb{P}(A))$, then $Q \ge_T (A_\alpha,\K(A_\alpha))$ for each $\alpha < \kappa$.
\end{lemma}

\begin{proof} We may assume each $A_\alpha$ is a subspace of $Y_\alpha$ a copy of $I^\N$. 
Let $Y = \bigoplus \{Y_\alpha : \alpha < \kappa\}$.  Define $A = \bigoplus \{A_\alpha : \alpha < \kappa\}$.  Then the map $\phi : P \to CL(Y) \cap \mathbb{P}(A)$ given by $\phi((K_\alpha)_\alpha) = \bigoplus_\alpha K_\alpha$ is order-preserving, and its image covers $A$, which proves (i).

Now suppose $Q \ge_T (A, CL(Y) \cap \mathbb{P}(A))$ is witnessed by $\phi : Q \to CL(Y) \cap \mathbb{P}(A)$.  Note that if $C$ is a subset of $A$ that is closed in $Y$, then $C \cap Y_\alpha$ is a subset of $A_\alpha$ that is closed in $Y_\alpha$, and since $Y_\alpha$ is compact, then so is $C \cap Y_\alpha$.  Thus, we have a map $\phi_\alpha : Q \to  \K(A_\alpha)$ given by $\phi_\alpha(q) = \phi(q) \cap Y_\alpha$ that witnesses (ii).
\end{proof}

\subsubsection{Covering and Base Properties}

As defined in the Introduction, a space is $P$-paracompact, for some directed set $P$, if every open cover has a $P$-locally finite open refinement. More generally, if $\mathcal{P}$ is a class of directed sets then a space is \emph{$\mathcal{P}$-paracompact} if  every open cover $\mathcal{U}$ has a $P$-locally finite open refinement $\mathcal{V}$ for some $P$ in $\mathcal{P}$ (depending on $\mathcal{U}$). These are our main covering properties but our study of $\K(M)$-paracompactness and $\K(\M)$-paracompactness lead to additional covering properties outlined below. 

Let $\mathcal{C}$ be a collection of subsets of a space. If for some cardinal $\kappa$ and property $\mathbf{P}$ we can write  $\mathcal{C}=\bigcup \{\mathcal{C}_\alpha : \alpha < \kappa\}$ where each $\mathcal{C}_\alpha$ has $\mathbf{P}$ then we say $\mathcal{C}$ is \emph{$\kappa$-$\mathbf{P}$}. Following tradition we say $\sigma$-$\mathbf{P}$ instead of $\aleph_0$-$\mathbf{P}$.

A collection $\mathcal{C}$ of subsets of a space is \emph{relatively locally finite} if it is locally finite in its union (so for each point $x$ in $\bigcup \mathcal{C}$ there is an open neighborhood meeting only finitely many elements of $\mathcal{C}$).
Given a cardinal  $\kappa$, a space is called \emph{$\kappa$-(relatively) paracompact} if every open cover  has a $\kappa$-(relatively)  locally finite open refinement.

Recall that a space is \emph{screenable} if every open cover has a $\sigma$-disjoint open refinement. Clearly a space with a $\sigma$-disjoint base is screenable.  
Observe that a pairwise disjoint collection of open sets is relatively locally finite, while a relatively locally finite family is point-finite. Hence:
\begin{lemma}\label{sdissrp_srpisspf}
Every $\sigma$-disjoint family of open sets is $\sigma$-relatively locally finite.

Every $\sigma$-relatively locally finite family of open sets is $\sigma$-point finite.
\end{lemma}

Let $\mathcal{C}$ be a family of subsets of a space $X$. Call $\mathcal{C}$ \emph{$(\omega_1,\omega)$-point finite} (respectively, \emph{$(\omega_1,\omega)$-locally finite}) if  every uncountable subcollection of $\mathcal{C}$ contains an infinite subcollection which is point-finite (respectively, locally finite). A space is \emph{$(\omega_1,\omega)$-paracompact} (respectively, \emph{$(\omega_1,\omega)$-metacompact}) if every open cover  has a  $(\omega_1,\omega)$-locally finite (respectively, $(\omega_1,\omega)$-point finite) open refinement. 
A space is \emph{$(\omega_1,\omega)$-metrizable}  if it has a  $(\omega_1,\omega)$-locally finite base.

By analogy with the property `weakly $\sigma$-point finite' used in \cite{GOO04}, we call a family $\mathcal{C}$ of subsets of a space $X$ \emph{weakly $\sigma$-locally finite} if we can write $\mathcal{C} = \bigcup_n \mathcal{C}_n$ in such a way that:
\begin{equation*} \label{wslf}
\forall x\in X,\ \bigcup\{\mathcal{C}_n : \mathcal{C}_n \text{ is locally finite at } x\} = \mathcal{C}. \tag{$\star$}
\end{equation*}
A space is \emph{weakly $\sigma$-paracompact} if every open cover has a weakly $\sigma$-locally finite open refinement.

\subsubsection{An Additional $P$-Property}

An open subset $U$ of a space is \emph{$P$-regular} if it can be written: $U=\bigcup \{V_q : q \in Q\} = \bigcup \{\cl{V_q} : q \in Q\}$, where all the $V_q$ are open sets and if $q \le q'$ then $V_q \subseteq V_{q'}$, where $Q=P \times \N$. Call a space \emph{$P$-perfectly normal} if every open subset is $P$-regular. For a class $\mathcal{P}$ of directed sets call a space $\mathcal{P}$-perfectly normal  if for every  open subset $U$ there is a $P$ in $\mathcal{P}$ such that $U$  is $P$-regular.  Note that a space is $1$-perfectly normal if and only if it is $\N$-perfectly normal if and only if it is perfectly normal (in the usual sense).

\subsubsection{Special Spaces and Subsets of $\R$}

A space $Y$ is a \emph{$Q$-space} if every subset is a $G_\delta$-subset. 
A subset $A$ of the reals, $\R$, is a \emph{$Q$-set} if and only if it is uncountable and a $Q$-space. 

A space $Y$ is a \emph{$\Delta$}-space if whenever we write $Y$ as an increasing union of subsets, $Y=\bigcup_n S_n$ where $S_n \subseteq S_{n+1}$ for all $n$, there is a
countable closed cover, $\{C_n:n \in \omega\}$, of $Y$, such that $C_n \subseteq S_n$ for every $n$. A subset $A$ of $\R$ is a \emph{$\Delta$-set} if it is an uncountable $\Delta$-space. It is clear that every $Q$-space is a $\Delta$-space.

Let $Y$ be a space and $A$ a subspace. Then $A$ is \emph{relatively countably compact} if every subset of $A$ which is closed discrete in $Y$ is finite. Call $Y$ \emph{RCCC} if every relatively countably compact subset of $Y$ is  countable. We note that a metrizable space $Y$ is RCCC if and only if every compact subset of $Y$ is countable (in other words, $Y$ is totally imperfect).

\subsection{Basic Results}
\subsubsection{Results Valid For All $P$}

Let $\mathcal{C}$ be  family of subsets of a space $X$. Let $P$ be a directed set. We say that $\mathcal{C}$ is \emph{$P$-ordered} if it can be written $\mathcal{C}=\bigcup \{ \mathcal{C}_p :  p \in P\}$ where if $p \le p'$ then $\mathcal{C}_p \subseteq \mathcal{C}_{p'}$. We call the $\mathcal{C}_p$ the \emph{levels} of the $P$-ordering. So $\mathcal{C}$ is $P$-locally finite precisely when it is $P$-ordered with locally finite levels.

Define $LF(\mathcal{C}) = \{\mathcal{L}\subseteq \mathcal{C} : \mathcal{L} \text{ is locally finite in } X\}$. Then $LF(\mathcal{C)}$ is a  Dedekind complete directed set. In the next lemma and subsequently  we abbreviate $([\mathcal{C}]^1,LF(\mathcal{C}))$ (where $[\mathcal{C}]^1=\{ \{C\} : C \in \mathcal{C}\}$) to $(\mathcal{C},LF(\mathcal{C}))$. 

\begin{lemma}\label{PlfviaTukey}
Let $\mathcal{C}$ be  family of subsets of a space $X$. Let $P$ be a directed set. Then the following are equivalent: 

(i) $\mathcal{C}$ is $P$-locally finite, (ii) there is an order-preserving map $\phi : P \to LF(\mathcal{C})$ such that $\phi(P)$ is cofinal for $\mathcal{C}$ in $LF(\mathcal{C})$, (iii) $P \ge_T (\mathcal{C},LF(\mathcal{C}))$, (iv) $P \ge_T ([\mathcal{C}]^{<\omega},LF(\mathcal{C}))$.

Hence if $\mathcal{C}$ is $P$-locally finite and $Q \ge_T P$ then $\mathcal{C}$ is $Q$-locally finite.
\end{lemma}
\begin{proof}
The equivalence of (i) and (ii) is immediate once the definitions are unpacked (in particular, `$\phi(P)$ is cofinal for $\mathcal{C}$ in $LF(\mathcal{C})$' means `$\bigcup \phi(P) = \mathcal{C}$'). Clearly (ii) $\implies$ (iii). Since $LF(\mathcal{C})$ is Dedekind complete, by Lemma~\ref{Ded} the converse is true. 

Clearly (iv) implies (iii). 
Now suppose (iii) holds. Then so does (ii), and there is an order-preserving $\phi : P \to LF(\mathcal{C})$ such that $\phi(P)$ is cofinal for $\mathcal{C}$ in $LF(\mathcal{C})$. We show $\phi$ witnesses (iv). To this end, take any $\{C_1, \ldots , C_n\}$ in $[\mathcal{C}]^{<\omega}$, and pick $p_1, \ldots , p_n$ in $P$ such that $C_i \in \phi(p_i)$ for $i=1, \ldots , n$. As $P$ is directed there is an upper bound, $p_0$, of $p_1, \ldots , p_n$. Then, as $\phi$ is order-preserving, $\phi (p_0) \supseteq \{C_1, \ldots, C_n\}$ --- as required. 

The last part follows from the equivalence of (i) and (iii) combined with transitivity of Tukey quotients.
\end{proof}

It follows that all our $P$-properties are respected by Tukey quotients.  
\begin{lemma} Suppose $Q \ge_T P$. 
If $X$ is $P$-paracompact (respectively, $P$-metrizable) then $X$ is $Q$-paracompact (respectively, $Q$-metrizable).
\end{lemma}
This allows us to `tidy' by replacing any given  directed set by anything equivalent or larger in the Tukey order. For example, a collection is $(P\times \N)$-locally finite if and only if it is $(P \times [\N]^{<\omega})$-locally finite. We use these simplifications without further comment.

\begin{lemma}
If $X$ is $P$-metrizable then it is $P$-perfectly normal.
\end{lemma}

\begin{proof}
Fix $\mathcal{B}=\bigcup\{\mathcal{B}_p : p \in P\}$ be a $P$-ordered base where each $\mathcal{B}_p$ is locally finite. Take any open subset $U$ of $X$. For $p$ in $P$ set $V_p=\bigcup \{ B : B \in \mathcal{B}_p$ and $\cl{B} \subseteq U\}$. Then by local finiteness of $\mathcal{B}_p$ we have $\cl{V_p} = \cl{ \bigcup \{ B : B \in \mathcal{B}_p \text{ and }\cl{B} \subseteq U\}} = \bigcup \{ \cl{B} : B \in \mathcal{B}_p,$ and $\cl{B} \subseteq U\} \subseteq U$. So the $V_p$ are as required for $P$-perfect normality.
\end{proof}

\begin{lemma}\label{srp_ppn_to_pp}
If $X$ is $\sigma$-relatively paracompact and $P$-perfectly normal (respectively, $\mathcal{P}$-perfectly normal), then $X$ is $P$-paracompact (respectively, $\mathcal{P}$-paracompact).
\end{lemma}

\begin{proof} We prove this for a fixed $P$. The class is very similar.
Let $\mathcal{U}$ be an open cover of $X$. We show it has a $P$-locally finite open refinement.
We know $\mathcal{U}$ has an open refinement $\mathcal{V}= \bigcup_n \mathcal{V}_n$ where $\mathcal{V}_n$ is locally finite in $V_n=\bigcup \mathcal{V}_n$, for each $n$. Then for each $n$, $V_n = \bigcup \{ W_p^n : p \in P\}=\bigcup \{\cl{W_p^n} : p \in P\}$ where for every $p$ the set $W_p^n$ is open and if $p \le p'$ then $W_p^n \subseteq W_{p'}^n$.

For a $p$ in $P$ and $n$ in $\N$ set $\mathcal{W}_{p,n} = \{ V \cap W_p^n : V \in \mathcal{V}_n\}$. Note that $\mathcal{W}_{p,n}$ is locally finite (because it is locally finite in $V_n$, which contains $\cl{W_p^n}$, and outside $\cl{W_p^n}$). For finite $F \subseteq \N$, let $\mathcal{W}_{p,F} = \bigcup_{n \in F} \mathcal{W}_{p,n}$. Then $\mathcal{W}=\bigcup \{ \mathcal{W}_{p,F} : p \in P, F \in [\N]^{<\omega}\}$ is a $(P \times [\N]^{<\omega})$-ordered open refinement of $\mathcal{U}$ with each level, $\mathcal{W}_{p,F}$, locally finite.
\end{proof}

\begin{lemma} \label{srlf_ppn_to_pm}
If $X$ has a $\sigma$-relatively locally finite base and is $P$-perfectly normal, then $X$ is $P$-metrizable.
\end{lemma}

\begin{proof}
Let $\mathcal{B}=\bigcup_n \mathcal{B}_n$ be a base for $X$ such that $\mathcal{B}_n$ is locally finite in $B_n=\bigcup \mathcal{B}_n$ for all $n$.  Then for each $n$, $B_n = \bigcup \{ W_p^n : p \in P\}=\bigcup \{\cl{W_p^n} : p \in P\}$ where for every $p$ the set $W_p^n$ is open and if $p \le p'$ then $W_p^n \subseteq W_{p'}^n$.

For a $p$ in $P$ and $n$ in $\N$ set $\hat{\mathcal{B}}_{p,n} = \{ B \cap W_p^n : B \in \mathcal{B}_n\}$. Note that $\hat{\mathcal{B}}_{p,n}$ is locally finite. For finite $F \subseteq \N$, let $\hat{\mathcal{B}}_{p,F} = \bigcup_{n \in F} \hat{\mathcal{B}}_{p,n}$. Then $\hat{\mathcal{B}}=\bigcup \{ \hat{\mathcal{B}}_{p,F} : p \in P, F \in [\N]^{<\omega}\}$ is a $(P \times [\N]^{<\omega})$-ordered collection of open sets with each level, $\mathcal{W}_{p,F}$, is  locally finite. Finally, $\hat{\mathcal{B}}$ is a base. For if $x$ is in $U$ open, then $x \in B \subseteq U$ for some $B$ in $\mathcal{B}_m$. Pick $p$ such that $x$ is in $W_p^m$. Then $x \in B'=B \cap W_p^m$, $B' \subseteq U$ and $B' \in \hat{\mathcal{B}}_{p,\{m\}} \subseteq \hat{\mathcal{B}}$.
\end{proof}

\begin{lemma} \label{ppms_to_pnm}
Suppose $\mathcal{P}$ is a class of directed sets that is countably directed with respect to the Tukey order, $\ge_T$.  If $X$ is a $\mathcal{P}$-paracompact Moore space, then $X$ is $P$-metrizable for some $P$ in $\mathcal{P}$.  In particular:

(i) if $X$ is a $P$-paracompact Moore space for some directed set $P$, then $X$ is $P$-metrizable, and

(ii) if $X$ is a $\K(\M)$-paracompact Moore space, then $X$ is $\K(M)$-metrizable for some separable metrizable $M$.
\end{lemma}

\begin{proof}
Let $\{\mathcal{G}_n : n\in\N\}$ be a development for $X$, and for each $n$ in $\N$, find a $P_n$-locally finite open refinement $\mathcal{U}_n$ of $\mathcal{G}_n$ for some $P_n$ in $\mathcal{P}$.  Since $\mathcal{P}$ is countably directed with respect to $\ge_T$, then there is a $P$ in $\mathcal{P}$ such that $P \ge_T P_n$ for every $n$.  Thus, each $\mathcal{U}_n$ is $P$-locally finite: $\mathcal{U}_n = \bigcup \{\mathcal{U}_{n,p} : p \in P\}$. Define $\mathcal{B}_{n,p} = \bigcup \{ \mathcal{U}_{i,p} : i \le n\}$.  Then $\mathcal{B} = \bigcup \{ \mathcal{B}_{n,p} : n\in\N, p \in P\}$ is a $(P \times \N)$-locally finite base for $X$.

Now note that (i) follows by taking $\mathcal{P} = \{P\}$, and (ii) also follows since Theorem~\ref{km_bounded} implies that $\K(\M)$ is countably directed with respect to $\ge_T$.
\end{proof}

\begin{lemma}
\label{kMpcptIscpct}
If $\mathcal{C}$ is $P$-locally finite then it is $\cof{P}$-locally finite.

Hence, if $X$ is $P$-paracompact then it is $\cof{P}$-paracompact.

Now suppose $\mathcal{P}$ is a class of directed sets such that for some cardinal $\kappa$ we have $\cof{P} \le \kappa$ for all $P$ in $\mathcal{P}$. If $X$ is $\mathcal{P}$-paracompact then $X$ is $\kappa$-paracompact.
\end{lemma}
\begin{proof}
Suppose that $\mathcal{C}$ is $P$-locally finite, say $\mathcal{C}=\bigcup \{\mathcal{C}_p : p \in P\}$ is a $P$-ordering where each level is locally finite. Let $Q$ be a cofinal subset of $P$ of cardinality $\cof{P}$. Then $\mathcal{C}=\bigcup \{ \mathcal{C}_q : q \in Q\}$, and we see that $\mathcal{C}$ is indeed the union of $\cof{P}$-many locally finite subcollections. This establishes the first claim. The remainder easily follows.
\end{proof}

\subsubsection{When $P$ Does Not Have Calibre $(\omega_1,\omega)$}

\begin{lemma}\label{general_P}
Let $P$ be a directed set which does not have calibre $(\omega_1, \omega)$. Let $X$ be any space with weight $\le \omega_1$. Then $X$ has a base $\mathcal{B} = \bigcup \{ \mathcal{B}_p : p \in P\}$ where (i) if $p \le p'$ then $\mathcal{B}_p \subseteq \mathcal{B}_{p'}$ and (ii) each $\mathcal{B}_p$ is finite, and hence $X$ is $P$-metrizable.
\end{lemma}

\begin{proof} Fix a basis $\mathcal{B}$ for $X$ of size $\le \omega_1$.
Since $P$ does not have calibre $(\omega_1,\omega)$ it contains a subset $S$ of size $\omega_1$ such that no infinite subset of $S$ has an upper bound. Fix a surjection $\alpha : S \to \mathcal{B}$.

For each $p$ in $P$ set $\mathcal{B}_p = \{ \alpha (s) : s \in   \down{p} \cap S\}$. Clearly if $p \le p'$ then $\mathcal{B}_p \subseteq \mathcal{B}_{p'}$. Since no infinite subset of $S$ has an upper bound, each $\mathcal{B}_p$ is finite. Lastly, as $\alpha$ is a surjection, $\mathcal{B}=\bigcup \{ \mathcal{B}_{s} : s \in S\} = \bigcup \{\mathcal{B}_p : p \in P\}$.
\end{proof}

\subsubsection{When $P$ Has Calibre $(\omega_1,\omega)$}

Simply expanding definitions yields:
\begin{lemma}  A collection $\mathcal{C}$ of subsets of a space $X$ is  $(\omega_1,\omega)$-locally finite if and only if $\mathcal{C}$ has relative calibre $(\omega_1,\omega)$  in $LF(\mathcal{C})$.  
\end{lemma}
The definitions of  $(\omega_1,\omega)$-point finite  can be restated similarly.

\begin{lemma}\label{Pcalom1om_imp_relom1om} Let $P$ be a directed set with calibre $(\omega_1,\omega)$.
Let $\mathcal{C}$ be a family of subsets of a space $X$.

If $\mathcal{C}$ is $P$-point finite then it is  $(\omega_1,\omega)$-point finite.  If $\mathcal{C}$ is $P$-locally finite then it is  $(\omega_1,\omega)$-locally finite.

Hence if $X$ is $P$-paracompact (respectively, $P$-metrizable) then it is $(\omega_1,\omega)$-paracompact (respectively, $(\omega_1,\omega)$-metrizable). 

Let $\mathcal{P}$ be a class of directed sets all with calibre $(\omega_1,\omega)$. If $X$ is $\mathcal{P}$-metacompact then it is $(\omega_1,\omega)$-metacompact.  If $X$ is $\mathcal{P}$-paracompact then it is $(\omega_1,\omega)$-paracompact.
\end{lemma}

\begin{proof}
If $\mathcal{C}$ is $P$-locally finite, then $P \ge_T (\mathcal{C},LF(\mathcal{C}))$.  Since $P$ has calibre $(\omega_1,\omega)$, then by taking $P'=P$ in Lemma~\ref{tukey_cal}, we see that $\mathcal{C}$ has relative calibre $(\omega_1,\omega)$ in $LF(\mathcal{C})$, which is equivalent to saying that $\mathcal{C}$ is  $(\omega_1,\omega)$-locally finite.

We can similarly prove that $P$-point finite collections are  $(\omega_1,\omega)$-point finite, and the remaining claims follow immediately.
\end{proof}

\begin{lemma} \label{com1om_family}
Let $\mathcal{C}$ be $(\omega_1,\omega)$-point finite. Then $\mathcal{C}$ is point countable.
\end{lemma}
\begin{proof} Suppose, for a contradiction, that there exists an $x$ in $X$ such that $\mathcal{C}_x=\{C: x\in C \text{ and } C\in \mathcal{C}\}$ is uncountable. But now every infinite subfamily of $\mathcal{C}_x$ contains $x$ in its intersection, and so is not point-finite. This indeed contradicts $\mathcal{C}$ being $(\omega_1,\omega)$-point finite.
\end{proof}
From the previous lemma we see that a given $(\omega_1,\omega)$-point finite base must be point-countable, and so every point has a countable local base.
\begin{lemma}\label{om1omMetis1o} 
Every space with a $(\omega_1,\omega)$-point finite base, in particular every $(\omega_1,\omega)$-metrizable space, is first countable.
\end{lemma}
\begin{lemma}\label{densesubset} 
Let $X$ be a space with a dense $\sigma$-compact subset. 
If $X$ is $(\omega_1,\omega)$-paracompact then $X$ is Lindel\"{o}f. If $X$ is  $(\omega_1,\omega)$-metrizable then $X$ is (separable) metrizable.
\end{lemma}

\begin{proof}
Note first that if $K$ is a compact subset of   $X$ and $\mathcal{U}$ is a locally finite
family of subsets of $X$, then there is an open $V$ containing $K$ which
meets only finitely many elements of $\mathcal{U}$. So if $\mathcal{U}$ is $(\omega_1,\omega)$-locally finite,  then every compact subset
of $X$ meets only countably many members of $\mathcal{U}$. 

It easily follows that if $X$ has a
dense $\sigma$-compact subset, then an $(\omega_1,\omega)$-locally finite open
refinement of a given open cover, or an $(\omega_1,\omega)$-locally finite base,  must be countable.
\end{proof}

For any cardinal $\kappa$, let $A(\kappa)$ denote the one-point compactification of $D(\kappa)$, the discrete space of size $\kappa$.

\begin{lemma} \label{supersequence}
If $\kappa > \omega$, then $A(\kappa)^2\setminus\diag$ is not $(\omega_1,\omega)$-paracompact.
\end{lemma}

\begin{proof}
Let $Y = A(\kappa)^2\setminus\diag$, and write $A(\kappa) = D(\kappa) \cup \{\infty\}$.  Consider the following open cover of $Y$:
\[\mathcal{U} = \{ (\{x\} \times A(\kappa)) \cap Y : x \in D(\kappa) \} \cup \{ Y \setminus (A(\kappa)\times\{\infty\}) \}.\]
Let $\mathcal{V}$ be any open refinement of $\mathcal{U}$.  We will show that $\mathcal{V}$ does not have relative calibre $(\omega_1,\omega)$ in $LF(\mathcal{V})$.

For each $x$ in $D(\kappa)$, choose $V_x \in \mathcal{V}$ such that $(x,\infty) \in V_x$.  As $\mathcal{V}$ refines $\mathcal{U}$, we have $V_x \subseteq \{x\} \times A(\kappa)$.  Then there is a finite set $F_x \subseteq D(\kappa)$ such that $V_x = \{x\} \times (A(\kappa) \setminus F_x)$.  Suppose $\mathcal{V}$ does have relative calibre $(\omega_1,\omega)$ in $LF(\mathcal{V})$.  Then there is a countably infinite subset $C \subseteq D(\kappa)$ such that $\{V_x : x \in C\}$ is locally finite.  Let $E = \bigcup \{F_x : x \in C\}$, which is countable, and choose any $y \in D(\kappa) \setminus E$.  Then $(x,y) \in V_x$ for each $x \in C$, so each neighborhood of $(\infty,y)$ intersects all but finitely many members of the infinite locally finite family $\{V_x : x \in C\}$, which is a contradiction.
\end{proof}

The following lemma can safely be left to the reader.
\begin{lemma} \label{card}
Let $X$ be a space and $x\in X$.  If $X\setminus\{x\}$ is $(\omega_1,\omega)$-paracompact or $P$-paracompact, then $X$ has the same property.
\end{lemma}

\subsubsection{The Case $P=\K(M)$}

\begin{prop}\label{base_local} Suppose that $X$ is first countable and $\mathcal{V}=\{\mathcal{V}_K: K\in \K(M)\}$ is a $\K(M)$-locally finite family of subsets of $X$.  Then for any $x\in X$ and $K\in \K(M)$, there exists an open neighborhood $T$ of $K$ such that $\mathcal{V}_T = \bigcup\{\mathcal{V}_L : L\in T\}$ is locally finite at $x$.
\end{prop}

\begin{proof} Suppose the conclusion is not true, which is to say we can find $x$ and $K$ such that for any neighborhood $T$ of $K$, the set $\mathcal{V}_T$ is not locally finite at $x$. Then let $\{B_m: m\in \N\}$ be a decreasing local base at $x$ and $\{T_n: n\in \N\}$ be a decreasing local base at $K$.  So for each $m\in \N$, the set $\{V\in \mathcal{V}_{T_n} : B_m\cap V\neq \emptyset\}$ is infinite for each $n\in\N$.

Fix $m\in\N$. Then for any $n\geq m$, inductively, we can find $V_n^m\in \mathcal{V}_{T_n}\setminus \{V_m^m, \ldots, V_{n-1}^m\}$ such that $B_m\cap V_n^m\neq \emptyset$. We can then find $K_n^m\in T_n$ such that $V_n^m\in \mathcal{V}_{K_n^m}$. 

Since $\{T_n:n\in \N\}$ is a decreasing local base at $K$, the set $\K=\{K_n^m: m,n\in\N,\ m\leq n\}\cup \{K\}$ is a compact subset of $\K(M)$.  Hence $\hat{K} = \bigcup \K$ is a compact subset of $M$ such that $K_n^m\subseteq \hat{K}$ whenever $m\le n$.

Then we have $\{V_n^m: m,n\in\N,\ m\le n\}\subseteq \mathcal{V}_{\hat{K}}$. But since for each $m\in\N$, $B_m$ intersects every $V_n^m$ with $n \ge m$, then $\mathcal{V}_{\hat{K}}$ is not locally finite at $x$, which is a contradiction.
\end{proof}

\begin{prop} \label{wslf_kmlf}
Let $\mathcal{V}$ be a collection of subsets of a space $X$.

If $\mathcal{V}$ is weakly $\sigma$-locally finite, then it is also $\K(M)$-locally finite for some separable metrizable $M$, and the converse is true if $X$ is first countable.
\end{prop}

\begin{proof}
Write $\mathcal{V} = \bigcup_n \mathcal{V}_n$ where (\ref{wslf}) is satisfied for each $x\in X$.  Now, for each $V \in \mathcal{V}$, define $\sigma_V \in \{0,1\}^\N$ by $\sigma_V(n) = 1$ if $V\in\mathcal{V}_n$ and $0$ if $V\not\in\mathcal{V}_n$.
Then let $M = \{\sigma_V : V\in\mathcal{V}\} \subseteq \{0,1\}^\N$.  For any compact subset $K$ of $M$, let $\mathcal{V}_K = \{V\in\mathcal{V} : \sigma_V \in K\}$.  Then $\mathcal{V} = \bigcup\{\mathcal{V}_K : K\in\K(M)\}$ and $\mathcal{V}_K \subseteq \mathcal{V}_L$ whenever $K,L\in\K(M)$ and $K\subseteq L$.

Now we check that each $\mathcal{V}_K$ is locally finite in $X$.  Fix $K$ in $\K(M)$ and define $U_n = \{\sigma\in M : \sigma(n) = 1\}$ for each $n\in\N$.  For any $x\in X$, we claim that $\{U_n : \mathcal{V}_n \text{ is locally finite at } x\}$ is an open cover of $K$.  Indeed, if $\delta$ is in $K$, then pick $V\in\mathcal{V}$ such that $\delta = \sigma_V$.  Then (\ref{wslf}) implies that there is an $m\in\N$ such that $V\in\mathcal{V}_m$ and $\mathcal{V}_m$ is locally finite at $x$.  It follows that $\delta$ is in $U_m$.

Then for each $x\in X$, there is a finite set $F_x\subseteq\N$ such that $K \subseteq \bigcup_{n\in F_x} U_n$ and $\mathcal{V}_n$ is locally finite at $x$ for each $n\in F_x$.  We now have:
\[\mathcal{V}_K \subseteq \left\{V\in\mathcal{V} : \sigma_V \in \bigcup_{n\in F_x} U_n\right\} = \bigcup_{n\in F_x} \mathcal{V}_n,\]
and so $\mathcal{V}_K$ is locally finite at $x$, and $\mathcal{V}$ is $\K(M)$-locally finite.

For the converse, write $\mathcal{V} = \bigcup\{\mathcal{V}_K : K\in\K(M)\}$ where each $\mathcal{V}_K$ is locally finite and $K\subseteq L$ implies $\mathcal{V}_K\subseteq \mathcal{V}_L$.  Fix a countable base $\B$ for $\K(M)$ and define $\mathcal{V}_B = \bigcup\{\mathcal{V}_K : K\in B\}$ for each $B\in\B$.  Then for each $x\in X$, Proposition~\ref{base_local} guarantees that $\bigcup\{\mathcal{V}_B : \mathcal{V}_B \text{ is locally finite at } x\} = \mathcal{V}$.  Hence, $\mathcal{V}$ is weakly $\sigma$-locally finite since $\mathcal{B}$ is countable.
\end{proof}

\begin{lemma} \label{wslf_to_srlf}
Let $\mathcal{V}$ be a weakly $\sigma$-locally finite family of subsets of a space $X$, and write $\mathcal{V}=\bigcup_n \mathcal{V}_n$ satisfying (\ref{wslf}) of the definition. For each $n$ in $\N$, define $X_n = \{x\in X : \mathcal{V}_n \text{ is locally finite at } x\}$ and  $\mathcal{W}_n = \{ V \cap X_n : V \in \mathcal{V}_n\}$. Define $\mathcal{W}=\bigcup_n \mathcal{W}_n$.

Then each $\mathcal{W}_n$ is relatively locally finite, so $\mathcal{W}$ is $\sigma$-relatively locally finite.  Moreover, if $\mathcal{V}$ is an open cover, then $\mathcal{W}$ is an open refinement of $\mathcal{V}$, and if $\mathcal{V}$ is a base for $X$, then so is $\mathcal{W}$.
\end{lemma}

\begin{proof}
Since $\bigcup\mathcal{W}_n$ is contained in $X_n$ and $\mathcal{V}_n$ is locally finite on $X_n$, then it follows that $\mathcal{W}_n$ is locally finite in its union $\bigcup\mathcal{W}_n$, which proves the first claim.

Note that each $X_n$ is open and $\mathcal{W}$ refines $\mathcal{V}$, so to prove the final two claims, it suffices to check that whenever $x\in V\in\mathcal{V}$, then there is a $W\in\mathcal{W}$ such that $x\in W\subseteq V$. Indeed, if $x\in V\in\mathcal{V}$, then by property (\ref{wslf}), there is an $n\in\N$ such that $V\in\mathcal{V}_n$ and $\mathcal{V}_n$ is locally finite at $x$.  Thus, $W = V\cap X_n$ is in $\mathcal{W}$ and $x \in W \subseteq V$.
\end{proof}

\subsection{Useful Constructions}
Let $Y$ be a space. Let $\Delta$ be the diagonal in $Y \times Y$. We introduce four `machines' for generating examples. We omit the proofs of some simple lemmas.

\paragraph{Machine 1 --- The $X$-machine}
Let $X(Y)$ have the underlying set $(Y \times Y \setminus \Delta) \cup Y$. Isolate all points of $Y^2 \setminus \Delta$. A basic open neighborhood of a point $y$ in $Y \subseteq X(Y)$ is $\{y\} \cup (U \times \{y\}) \cup (\{y\} \times U)$ for any open neighborhood $U$ of $y$ in $Y$.
Note that $X(\R)$ is (homeomorphic) to $\R \times \R$ with points away from the $x$-axis isolated, and points on the $x$-axis have neighborhoods in the shape of an `X' --- and the closed upper half plane of this latter space is Heath's V-space. In other words, $X(\R)$ is a symmetric version of Heath's V-space.

\begin{lemma} \label{X_basic_facts}
For any Hausdorff space $Y$, $X(Y)$ is zero-dimensional and Hausdorff and hence Tychonoff.

The $X$-space, $X(Y)$, is a Moore space if and only if $Y$ is first countable.
Hence, for any directed set $P$, $X(Y)$ is $P$-metrizable if and only if $Y$ is first countable and $X(Y)$ is $P$-paracompact.
\end{lemma}

\begin{lemma} \label{RCCC}
Let $Y$ be any space. If $Y$ is RCCC, then $X(Y)$ is $(\omega_1,\omega)$-paracompact. Hence, if $Y$ is RCCC and first countable, then $X(Y)$ is $(\omega_1,\omega)$-metrizable. 
If $Y$ is metrizable and $X(Y)$ is $(\omega_1,\omega)$-paracompact then $Y$ is RCCC.
\end{lemma}

\begin{proof}
Any open cover of $X(Y)$ has an open refinement of the form $\mathcal{U} = \mathcal{U}_1 \cup \mathcal{U}_2$, where $\mathcal{U}_1$ contains one basic open neighborhood $U_y$ for each point $y\in Y \subseteq X(Y)$, and $\mathcal{U}_2$ consists of the singletons for each point in $X(Y)$ not already covered by $\mathcal{U}_1$.  Notice that $\mathcal{U}_2$ is locally discrete.  Now assume $Y$ is RCCC. Then to show $X(Y)$ is $(\omega_1,\omega)$-paracompact, it suffices to show that $\mathcal{U}_1$ is $(\omega_1,\omega)$-locally finite.

Suppose $\mathcal{V}$ is an uncountable subset of $\mathcal{U}_1$, so $\mathcal{V} = \{U_y : y\in A\}$ for some uncountable $A\subseteq Y$.  Since $Y$ is RCCC, then $A$ is not relatively countably compact in $Y$, so there is an infinite subset $S$ of $A$ that is closed and discrete in $Y$.  It is then easy to check that $\mathcal{W} = \{U_y : y\in S\}$ is an infinite locally finite subset of $\mathcal{V}$.
Thus, the first claim has been proven, and the second claim follows from Lemma~\ref{X_basic_facts}.

To prove the final claim, fix a metric generating the topology on $Y$, and for any $y\in Y$ and $n\in\N$, let $B_n(y)$ denote the open ball of radius $\frac{1}{n}$ centered at $y$.  Assuming $X(Y)$ is $(\omega_1,\omega)$-paracompact, then for each $y\in Y$, we can find an $n_y\in\N$ such that $\{ U_y = \{y\}\cup(B_{n_y}(y)\times\{y\})\cup(\{y\}\times B_{n_y}(y)) : y\in Y\}$ is $(\omega_1,\omega)$-locally finite in $X(Y)$.  Let $A$ be any uncountable subset of $Y$. By counting, there is an uncountable subset $A_1$ of $A$ and an $m\in\N$ such that $n_y = m$ for all $y\in A_1$.  Then there is an infinite subset $S$ of $A_1$ such that $\{U_y : y\in S\}$ is locally-finite.

We claim $S$ is closed and discrete in $Y$, which shows that $A$ is not relatively countably compact in $Y$, so that $Y$ is RCCC.  To that end, suppose  some $z\in Y$ is in the closure of $S\setminus\{z\}$.  So any basic neighborhood $B_n(z)$ of $z$ in $Y$ contains an element $y_n$ of $S\setminus\{z\}$.  Hence, any basic neighborhood $\{z\}\cup(B_n(z)\times\{z\})\cup(\{z\}\times B_n(z))$ of $z$ in $X(Y)$ intersects $U_{y_k}$ for each $k\ge\max\{n,m\}$, contradicting the fact that $\{U_y : y\in S\}$ is locally finite.
\end{proof}

\begin{lemma} \label{not_rp}
Let $Y$ be a space such that $w(Y) < |Y|$. Then $X(Y)$ is not $w(Y)$-relatively paracompact. 
\end{lemma}

\begin{proof}
Fix a base $\B$ for $Y$ with cardinality $\kappa = w(Y)$, and suppose $X(Y)$ is $\kappa$-relatively paracompact.  Then there is a collection $\mathcal{U} = \{U_y : y\in Y\} = \bigcup \{\mathcal{U}_\alpha : \alpha<\kappa\}$ where each $U_y$ is a basic neighborhood of $y$ in $X(Y)$ and each $\mathcal{U}_\alpha$ is locally finite in its union.

Since $|Y|>\kappa$, there is a $\beta<\kappa$ such that $Y_\beta = Y \cap (\bigcup\mathcal{U}_\beta)$ has size greater than $\kappa$.  Note that $\{U_y : y\in Y_\beta\} = \mathcal{U}_\beta$ is locally finite on $Y_\beta \subseteq  X(Y)$, so by shrinking the elements of $\mathcal{U}_\beta$, we can obtain a collection $\mathcal{V} = \{V_y : y\in Y_\beta\}$ where each $V_y$ is a basic $X(Y)$-neighborhood of $y$ that intersects only finitely many other members of $\mathcal{V}$.  In fact, any two distinct members of $\mathcal{V}$ will intersect in at most two points, so without loss of generality, $\mathcal{V}$ is actually pairwise disjoint.

For each $y\in Y_\beta$, write $V_y = \{y\}\cup(B_y\times\{y\})\cup(\{y\}\times B_y)$ for some $B_y\in\B$.  Then there is a $B\in\B$ and a subset $S$ of $Y_\beta$ such that $|S|>\kappa$ and $B_y = B$ for every $y\in S$. Note that $S$ is a subset of $B$.  Now pick any two distinct points $y_1$ and $y_2$ in $S$.  Then the point $(y_1,y_2)$ is in the intersection of $V_{y_1}$ and $V_{y_2}$, contradicting that $\mathcal{V}$ is pairwise disjoint.
\end{proof}

For any space $Y$ denote by $Y_\omega$ the space with underlying set $Y$ and topology obtained by adding all co-countable subsets of $Y$ to the original topology on $Y$. Note that if the original topology on $Y$ is Hausdorff then so is $Y_\omega$.

\begin{lemma} \label{co_countable}
Let $Y$ be a space. 

(i) $X(Y_\omega)$ is $P$-paracompact where $P$ is the directed set $[Y]^{\le \omega}$.

(ii) If $w(Y) \cdot \omega_1 < |Y|$, then $X(Y_\omega)$ is not $w(Y)$-relatively paracompact.
\end{lemma}

\begin{proof}
To prove (i), it suffices to show that any open cover for $X(Y_\omega)$ of the form $\mathcal{U} = \mathcal{U}_1 \cup \mathcal{U}_2$ as in the proof of Lemma~\ref{RCCC} is $P$-locally finite.  Write $\mathcal{U}_1 = \{U_y : y\in Y_\omega\}$ where each $U_y$ is a basic open neighborhood of $y$ in $X(Y_\omega)$.  Then for any countable subset $C$ of $Y$, we have that $C$ is closed and discrete in $Y_\omega$, so the family $\mathcal{U}_C = \{U_y : y\in C\} \cup \mathcal{U}_2$ is locally finite in $X(Y_\omega)$.  Thus, $\mathcal{U} = \bigcup \{\mathcal{U}_C : C\in P = [Y]^{\le\omega}\}$ is $P$-locally finite.

For the proof of (ii), assume $X(Y_\omega)$ is $\kappa$-relatively paracompact, where $\kappa = w(Y)$, and fix a base $\B$ for $Y$ with size $\kappa$.  Then $\{B \setminus C : B\in\B,\ C \subseteq Y,\ |C|\le\omega\}$ is a base for $Y_\omega$.  By only slightly modifying the proof of Lemma~\ref{not_rp}, we can find a $B\in\B$, a subset $S$ of $B$ with $|S|> \kappa\cdot\omega_1$, and countable sets $C_y \subseteq Y \setminus \{y\}$ for each $y\in S$ such that the collection $\mathcal{V} = \{V_y = \{y\}\cup((B\setminus C_y)\times\{y\}) \cup (\{y\}\times (B\setminus C_y)) : y\in S\}$ is pairwise disjoint.

Choose an arbitrary subset $A$ of $S$ with size $\omega_1$, and let $A' = A\cup(\bigcup\{C_y : y\in A\})$, which also has size $\omega_1$.  Then there is a $y_1$ in $S\setminus A'$ and a $y_2$ in $A\setminus C_{y_1}$.  Hence, we have $y_1\in B\setminus C_{y_2}$ and $y_2\in B\setminus C_{y_1}$, which means $V_{y_1}$ intersects $V_{y_2}$, which is a contradiction.
\end{proof}

\paragraph{Machine 2 --- The Split $X$-machine}

For any point $y$ in $Y$, write $y^+$ for $(y,+)$ and $y^-$ for $(y,-)$. For any subset $S$ of $Y$, let $S^+=\{s^+ : s \in S\}$ and $S^-=\{s^- : s \in S\}$.
Let $S(Y)$ have the underlying set $(Y \times Y \setminus \Delta) \cup Y^+ \cup Y^-$. Isolate all points of $Y^2 \setminus \Delta$. A basic open neighborhood of a point $y^+$ in $S(Y)$ is $\{y^+\} \cup (U \times \{y\})$ for any open neighborhood $U$ of $y$ in $Y$. A basic open neighborhood of a point $y^-$ in $S(Y)$ is $\{y^-\} \cup (\{y\} \times U)$ for any open neighborhood $U$ of $y$ in $Y$.
Note that $S(\R)$ is a symmetric version of Heath's split V-space.

\begin{lemma}
For any Hausdorff space $Y$, $S(Y)$ is zero-dimensional and Hausdorff and hence Tychonoff.

The split $X$-space, $S(Y)$, is a Moore space if and only if $Y$ is first countable. 
Hence, for any directed set $P$, $S(Y)$ is $P$-metrizable if and only if $Y$ is first countable and $S(Y)$ is $P$-paracompact.
\end{lemma}

\begin{lemma}
For any space $Y$, the split $X$-space, $S(Y)$, is screenable and therefore $\sigma$-relatively paracompact.
The space $S(Y)$ has a $\sigma$-disjoint base if and only if it has a $\sigma$-relatively locally finite base if and only if $Y$ is first countable. 
\end{lemma}

\begin{lemma}\label{sy_om1om_pcpt}
Let $Y$ be any space. If $Y$ is RCCC, then $S(Y)$ is $(\omega_1,\omega)$-paracompact. Hence, if $Y$ is RCCC and first countable, then $S(Y)$ is $(\omega_1,\omega)$-metrizable.
If $Y$ is metrizable and $S(Y)$ is $(\omega_1,\omega)$-paracompact, then $Y$ is RCCC.
\end{lemma}

\begin{proof}
The proof for Lemma~\ref{RCCC} can be easily modified to work here.
\end{proof}

For $Y$ a subset of $\R$, define $H(Y)=\{ (y,y') : y < y', \ y,y' \in Y\} \cup Y^+ \cup Y^-$ with the subspace topology from  $S(Y)$. For $y$ in $Y$ and $n$ in $\N$, define $V_n(y,+) = \{y^+\} \cup (y-1/n,y) \times \{y\}$ and $V_n(y,-)= \{y^-\} \cup \{y\} \times (y,y+1/n)$. These are basic neighborhoods of $y^+$ and $y^-$, respectively. 
As alluded to above, Heath's split V-space is (homeomorphic to) the subspace $H=H(\R)$. This family of subspaces  has some specific properties we identify.

\begin{lemma}\label{Hom1ommet} For any subspace $Y$ of $\R$, 
the space $H(Y)$ is $(\omega_1,\omega)$-metrizable.
\end{lemma}

\begin{proof} 
Let  $\mathcal{B} = \{\{(y,y')\} : y < y', \ y,y' \in Y\} \cup \{V_n(y,+) : y \in Y, \ n \in \N\} \cup \{  V_n(y,-) : y \in Y, \ n \in \N\}$. This is a basis for $H(Y)$. We show it is $(\omega_1,\omega)$-locally finite.

Let $\B_1$ be any uncountable subset of $\B$.  There must be an $n\in\N$ and an uncountable subset $\B_2$ of $\B_1$ as in one of the following three cases.

\textbf{Case 1:}  Each element of $\B_2$ is a singleton of the form $\{(y,y')\}$, where $y+\frac{1}{n} < y'$.  Then $\B_2$ is clearly  locally finite in $H(Y)$.

\textbf{Case 2:}  $\B_2 = \{V_n(y,+) : y \in Y'\}$ for some uncountable $Y' \subseteq Y$.  Since $\R$ with the `left' Sorgenfrey topology (in other words, with base $\{(a,b] : a < b\}$) has countable extent, then $Y'$ contains a strictly increasing sequence $(y_k)_k$ that converges in $\R$.  It is then straightforward to check that $\B_3 = \{V_n(y_k,+) : k\in\N\}$ is locally finite in $H(Y)$.

\textbf{Case 3:}  $\B_2 = \{V_n(y,-) : y \in Y'\}$ for some uncountable $Y' \subseteq Y$.  A similar argument (using the `right' Sorgenfrey topology and extracting a strictly decreasing convergent sequence) as for case 2 works here.

In any case, $\B_1$ contains an infinite locally finite subset, so the proof is complete.
\end{proof}

\begin{lemma}\label{HnotPpcpt} For any subspace $Y$ of $\R$ that is not RCCC, the space $H(Y)$ is not $P$-paracompact for any $P$ with calibre $(\omega_1,\omega)$.
\end{lemma}

\begin{proof} Fix a subspace $Y$ of $\R$, an uncountable relatively countably compact subset $A$ of $Y$, and a directed set $P$ with calibre $(\omega_1,\omega)$.  To get a contradiction, suppose $\B$ is a $P$-locally finite base for $H(Y)$.  Then according to Lemma~\ref{PlfviaTukey}, we have $P \tr ([\B]^{<\omega}, LF(\B))$, and so by Lemma~\ref{tukey_cal}, $[\B]^{<\omega}$ has relative calibre $(\omega_1,\omega)$ in $LF(\B)$.

For each $y\in A$, there are $W_y^+, W_y^- \in \B$ and an $n_y\in\N$ such that $V_{n_y}(y,+) \subseteq W_y^+ \subseteq V_1(y,+)$ and $V_{n_y}(y,-) \subseteq W_y^- \subseteq V_1(y,-)$.  We can find an uncountable subset $A_1$ of $A$ and an $n\in\N$ such that $n_y = n$ for all $y\in A_1$.  Then the uncountable subset $\{ \{W_y^+,W_y^-\} : y \in A_1\}$ of $[\B]^{<\omega}$ must contain an infinite subset with an upper bound in $LF(\B)$.  Thus, there exists an infinite $A_2 \subseteq A_1$ such that $\mathcal{W} = \bigcup \{ \{W_y^+,W_y^-\} : y \in A_2\}$ is locally finite.

If $A_2$ contains an increasing sequence that converges to some point $a$ in $A$, then $\mathcal{W}$ fails to be locally finite at $a^+$, since $V_n(y,-) \subseteq W_y^+$ for each $y \in A_2$.  Thus, $A_2$ does not contain any increasing sequence that converges in $Y$, and similarly, we can show $A_2$ does not contain any decreasing sequence that converges in $Y$.  But that contradicts the fact that $A$ is relatively countably compact in $Y$.
\end{proof}

\paragraph{Machine 3 --- The Disjoint Sets Split $X$-machine}

This is just a subspace of the previous split $X$-space, but it is so useful we isolate it. Let $Y$ be a space, and $A_1$, $A_2$ be a partition of $Y$.
Let $D(Y;A_1,A_2)$ be the subspace $(Y^2\setminus \diag) \cup A_1^+ \cup A_2^-$ of $S(Y)$.

\begin{lemma} \label{D_basics}
For any Hausdorff space $Y$ and partition $A_1, A_2$, the disjoint sets split-$X$ space, $D(Y;A_1,A_2)$ is zero-dimensional and Hausdorff and hence Tychonoff.

Further, $D(Y;A_1,A_2)$ is a Moore space if and only if $Y$ is first countable.
Hence, $D(Y;A_1,A_2)$ is $P$-metrizable if and only if $Y$ is first countable and $D(Y;A_1,A_2)$ is $P$-paracompact.
\end{lemma}

\begin{lemma}
For any space $Y$ and partition $A_1, A_2$, the space $D(Y;A_1,A_2)$ is screenable and hence $\sigma$-relatively paracompact.

The space $D(Y;A_1,A_2)$ has a $\sigma$-disjoint base if and only if it has a $\sigma$-relatively locally finite base if and only if $Y$ is first countable. 
\end{lemma}

\begin{lemma}\label{Dpmet}
Let $Y$ be a space with partition $A_1,A_2$. Let $P$ be a directed set.  If $P \ge_T (A_i, \mathop{CL} (Y) \cap \mathbb{P}(A_i))$ for $i=1,2$, then $D(Y;A_1,A_2)$ is $P$-paracompact.
If $Y$ is metrizable and $D(Y;A_1,A_2)$ is $P$-paracompact, then $P \times \mathbb{N} \ge_T (A_i, \mathop{CL} (Y) \cap \mathbb{P}(A_i))$ for $i=1,2$.
\end{lemma}

\begin{proof}
Suppose $P \ge_T (A_i, \mathop{CL} (Y) \cap \mathbb{P}(A_i))$ for $i=1,2$.  Then for each $i=1,2$, there is an order-preserving map $\phi_i: P \to \mathop{CL} (Y) \cap \mathbb{P}(A_i)$ whose image covers $A_i$.  Each open cover of $D(Y;A_1,A_2)$ has an open refinement of the form $\mathcal{U} = \mathcal{U}_0 \cup \mathcal{U}_1 \cup \mathcal{U}_2$, where $\mathcal{U}_1$ contains one basic neighborhood $U_x$ of each $x$ in $A_1^+$, $\mathcal{U}_2$ contains one basic neighborhood $U_x$ of each $x$ in $A_2^-$, and $\mathcal{U}_0$ contains the singletons for each point not covered by $\mathcal{U}_1\cup\mathcal{U}_2$.

For each $p\in P$, let $\mathcal{U}_p = \{U_x : x\in \phi_1(p) \cup \phi_2(p)\} \cup \mathcal{U}_0$.  Note that $\mathcal{U}_1$ is trivially locally finite in $D(Y;A_1,A_2) \setminus A_2^-$, and since $\phi_1(p) \subseteq A_1$ is closed in $Y$, then $\{U_x : x\in\phi_1(p)\}$ is locally finite in all of $D(Y;A_1,A_2)$.  Similarly, the rest of $\mathcal{U}_p$ is locally finite also.  Thus, $\bigcup \{\mathcal{U}_p : p\in P\} = \mathcal{U}$ is $P$-locally finite, which proves the first claim.

For the second claim, fix a metric for $Y$.  Then for any $y\in Y$, let $B_n(y)$ denote the open ball of radius $\frac{1}{n}$ centered at $y$, and define $U_n(y) = \{y^+\} \cup (B_n(y) \times \{y\})$ when $y\in A_1$ and $U_n(y) = \{y^-\} \cup (\{y\} \times B_n(y))$ when $y\in A_2$.  Since $D(Y;A_1,A_2)$ is $P$-paracompact, then for any $y\in Y$, there is an $n_y\in\N$ such that $\mathcal{U} = \{U_{n_y}(y) : y\in Y\}$ is $P$-locally finite.  Write $\mathcal{U} = \bigcup \{\mathcal{U}_p : p\in P\}$ where each $\mathcal{U}_p$ is locally finite and $\mathcal{U}_p \subseteq \mathcal{U}_{p'}$ whenever $p\le p'$ in $P$.

Fix $p\in P$ and $n\in\N$ and define $A_{i,p,n} = \{y\in A_i : U_{n_y}(y) \in \mathcal{U}_p,\ n_y \le n\}$.  We will check that $\cl{A_{i,p,n}}^Y$ is contained in $A_i$.  Let $a$ be in $Y\setminus A_i = A_{3-i}$. Then $a$ has a basic neighborhood $U_k(a)$ with $k\ge n$ that intersects $U_{n_y}(y)$ for only finitely many $y$ in $A_{i,p,n}$, and since those basic neighborhoods intersect in only one point, then we can actually assume $U_k(a)$ does not intersect any $U_{n_y}(y)$ for $y\in A_{i,p,n}$.  We claim that $B_k(a)$ does not intersect $A_{i,p,n}$, which shows that $y$ is not in $\cl{A_{i,p,n}}^Y$.  Indeed, if there were some $y$ in $A_{i,p,n}\cap B_k(a)$, then $a$ would be in $B_k(y) \subseteq B_{n_y}(y)$, so either $(a,y)$ or $(y,a)$ would be in $U_{n_y}(y)\cap U_k(a)$, which is a contradiction.

Then it is straightforward to check that $\phi_i: P\times\N \to CL(Y)\cap\mathbb{P}(A_i)$ defined by $\phi_i(p,n) = \cl{A_{i,p,n}}^Y$ is order-preserving and its image covers $A_i$, so $P \times \mathbb{N} \ge_T (A_i, \mathop{CL} (Y) \cap \mathbb{P}(A_i))$.
\end{proof}

\begin{lemma}\label{Dctblypcpt}
If $Y$ is a $\Delta$-space, then $D(Y;A_1,A_2)$ is countably paracompact.
\end{lemma}

\begin{proof}
Let $Y$ be a $\Delta$-space.  Recall that a space $X$ is countably paracompact if and only if for every countable increasing open cover $\{U_n : n\in\N\}$ of $X$, there is an open cover $\{V_n : n\in\N\}$ of $X$ such that $\cl{V_n}\subseteq U_n$ for each $n\in\N$.  So, let $\{U_n : n\in\N\}$ be an increasing open cover of $D(Y;A_1,A_2)$.

For each $n\in\N$ and $i=1,2$, let $S_n^i \subseteq A_i$  such that $(S_n^1)^+ = U_n\cap A_1^+$ and $(S_n^2)^- = U_n\cap A_2^-$, and then define $S_n = S_n^1\cup S_n^2$.  Then $\{S_n : n\in\N\}$ forms an increasing cover of $Y$, so there is a closed cover $\{C_n : n\in\N\}$ of $Y$ such that $C_n\subseteq S_n$ for each $n\in\N$.  Let $C_n^i = C_n\cap A_i$ for each $n\in\N$ and $i=1,2$, so $C_n^i \subseteq S_n^i$ and  $\bigcup_n C_n^i = A_i$.

For each $y$ in $C_n^1$, pick a basic open set $B_{y,n}$ such that $y^+ \in B_{y,n} \subseteq U_n$, and for each $y$ in $C_n^2$, pick a basic open $B_{y,n}$ such that $y^- \in B_{y,n} \subseteq U_n$.  Let $W_n^i = \bigcup \{ B_{y,n} : y \in C_n^i\}$ for $i=1,2$. Then $W_n^i$ is open and contained in $U_n$.

Notice also that if $z\in\cl{W_n^1}\setminus W_n^1$, then $z = a^-$ for some $a$ in $A_2 \cap \cl{C_n^1} \subseteq A_2 \cap C_n = C_n^2$.  Thus, $\cl{W_n^1}$ is contained in $W_n^1 \cup (C_n^2)^- \subseteq U_n$.  Similarly, we also have $\cl{W_n^2} \subseteq U_n$.  Let $W_n = W_n^1 \cup W_n^2$. Then $W_n$ is open,  its closure is contained in $U_n$, and $\bigcup_n W_n$ contains $A_1^+ \cup A_2^-$.

Let $T = D(Y,A_1,A_2) \setminus \bigcup_n W_n$, which is a clopen set of isolated points.  Then $T_n = T \cap U_n$ is also clopen for each $n$.  Let $V_n = W_n \cup T_n$.  Then $\{V_n : n\in\N\}$ is an open cover of $D(Y;A_1,A_2)$ such that $\cl{V_n} \subseteq U_n$.  Thus, $D(Y;A_1,A_2)$ is countably paracompact.
\end{proof}

\paragraph{Machine 4 --- Generalized Michael Line}

Let $Y$ be a metrizable space and $A$ any subset. Then $M(Y,A)$ has  $Y$ as its underlying set and is topologized by adding all singletons, $\{a\}$, for $a$ in $A$, to the original topology on $Y$.

\begin{lemma}\label{Mpmet}
Fix a metrizable space $Y$ and a subspace $A$. Let $P$ be a directed set. Then $M(Y,A)$ is $P$-metrizable if and only if $(P\times \N) \ge_T (A, \mathop{CL}(Y) \cap \mathbb{P}(A))$.
\end{lemma}

\begin{proof}
Let $Q=P \times \N$. Suppose, first, that $M(Y,A)$ is $P$-metrizable.  Then $M(Y,A)$ has a $Q$-ordered base $\mathcal{B}=\bigcup \{\mathcal{B}_q : q \in Q\}$ where every $\mathcal{B}_q$ is locally finite. For each $q$, let $\mathcal{B}_q^A = \mathcal{B}_q \cap \{ \{a\} : a \in A\}$, and $B_q = \bigcup \mathcal{B}_q^A$. Since all points of $A$ are isolated, the $\mathcal{B}_q^A$ form a $Q$-ordered clopen cover of $A$ by families locally finite in $M(Y,A)$. Hence the $B_q$ form a $Q$-ordered cover of $A$ by sets closed in $M(Y,A)$. By definition of the topology on $M(Y,A)$, the closure in $Y$ (with its original topology) of a $B_q$, call it $C_q$, is contained in $A$. Hence, the family $\{C_q : q \in Q\}$ witness that $Q \ge_T (A, \mathop{CL}(Y) \cap \mathbb{P}(A))$.

Now suppose $\{ C_q : q \in Q\}$ is a $(P\times \N)$-ordered cover of $A$ by subsets of $A$ which are closed in $Y$. Let $\mathcal{B}'=\bigcup_n \mathcal{B}_n'$ be a base for $Y$ (with its original, metrizable topology) such that $\mathcal{B}_n' \subseteq \mathcal{B}_m'$ when $n \le m$ and each $\mathcal{B}_n'$ is locally finite. Define $\mathcal{B}=\bigcup \{ \mathcal{B}_{q,n} : q \in Q, n \in \N\}$ by $\mathcal{B}_{q,n} = \mathcal{B}_n' \cup \{ \{a\} : a \in C_q\}$. Since $Q\times \N =_T P \times \N =Q$, it is easy to see that $\mathcal{B}$  shows  $M(Y,A)$ is $P$-metrizable.
\end{proof}

Suppose $Y$ is \emph{compact} metrizable and $A$ is any subspace of $Y$. Then $\mathop{CL}(Y)=\K(Y)$, and $\mathop{CL}(Y) \cap \mathbb{P}(A) = \K(A)$. We deduce:

\begin{lemma}
Let $Y$ be compact metrizable, and let $A$ be a subset of $Y$. Then $M(Y,A)$ is $\K(A)$-metrizable. If $M(Y,A)$ is $P$-metrizable for some directed set $P$, then $P\times\N \ge_T (A,\K(A))$.
\end{lemma}

\section{Main Results}

\subsection{(Pseudo)Compact $X$ and $X^2\setminus\diag$}

We recall that a countably compact space is compact if it is metaLindel\"{o}f and metrizable if it has a point-countable base \cite{Mis}. Further, a pseudocompact space is compact if $\sigma$-metacompact and metrizable if it has a $\sigma$-point finite base \cite{Usp}. On the other hand there are pseudocompact spaces with a point-countable base (hence metaLindel\"{o}f) which are not compact (and so not metrizable) \cite{Shak}.
\begin{lemma} \label{pseudo}
Let $X$ be  a pseudocompact space. If $X$ is  $(\omega_1,\omega)$-paracompact, then $X$ is compact. If $X$ is $(\omega_1,\omega)$-metrizable then $X$ is metrizable.
\end{lemma}

\begin{proof}
Recall that $X$ is pseudocompact if and only if each locally finite family of open subsets of $X$ is finite.  Thus, any open cover $\mathcal{V}$ of $X$ with relative calibre $(\omega_1,\omega)$ in $LF(\mathcal{V})$ must be countable. Both claims are now immediate.
\end{proof}

The following  natural questions are open:
\begin{ques}
Let $X$ be a pseudocompact space.

Is $X$ compact if (i) $(\omega_1,\omega)$-metacompact, or (ii) $P$-metacompact for some $P$ with calibre $(\omega_1,\omega)$, or (iii) $\K(M)$-metacompact?

Is $X$ metrizable if it has (i) an $(\omega_1,\omega)$-point finite base, or (ii) a $P$-point finite base for some $P$ with calibre $(\omega_1,\omega)$, or (iii) a $\K(M)$-point finite base?
\end{ques}

Recall that Gruenhage showed that a compact space $X$ is metrizable if and only if $X^2\setminus\diag$ is paracompact. We prove an optimal $P$-paracompact variant. The following lemma is extracted from Gruenhage's proof.

\begin{lemma}[Gruenhage, \cite{Gru84}] \label{GrLemma}
Let $X$ be compact and not metrizable.  If $X^2\setminus\diag$ has a partition $\{S_\alpha : \alpha < \kappa\}$ such that each $S_\alpha$ is open in $X^2\setminus\diag$ and Lindel\"{o}f, then $X$ contains a subspace homeomorphic to $A(\kappa)$ for some uncountable $\kappa$.
\end{lemma}

\begin{proof}
Since each $S_\alpha$ is Lindel\"{o}f and open in $X^2\setminus\diag$, we can write $S_\alpha = \bigcup_{n\in\N} U_{\alpha,n} \times V_{\alpha,n}$, where $U_{\alpha,n}$ and $V_{\alpha,n}$ are disjoint open sets in $X$ for each $n\in\N$.  Define $\mathcal{W} = \{U_{\alpha,n} : \alpha<\kappa,\ n\in\N\} \cup \{V_{\alpha,n} : \alpha<\kappa,\ n\in\N\}$.  
Then $\mathcal{W}$ is a $T_2$-separating open cover of $X$.
Since any compact space with a point-countable $T_1$-separating open cover is metrizable \cite{Mis}, and by hypothesis $X$ is not metrizable, $\mathcal{W}$ cannot be point-countable.  Hence, there is a point $x \in X$ contained in uncountably many members of $\mathcal{W}$.  Without loss of generality, there is an uncountable subset $A \subseteq \kappa$ and an $m\in\N$ such that $x \in \bigcap_{\alpha \in A} U_{\alpha,m}$.
Because $U_{\alpha,m} \times V_{\alpha,m} \subseteq S_\alpha$, then $\{U_{\alpha,m}\times V_{\alpha,m} : \alpha \in A\}$ is a discrete collection in $X^2\setminus\diag$.  It follows that $\{V_{\alpha,m} : \alpha\in A\}$ is a discrete collection in $X\setminus\{x\}$.
Thus, if we choose a point $y_\alpha \in V_{\alpha,m}$ for each $\alpha\in A$, then $Y = \{y_\alpha : \alpha\in A\}$ is an uncountable closed discrete subspace of $X\setminus\{x\}$.  As $X$ is compact, $\cl{Y}^X = \{x\} \cup Y$ is the one-point compactification of $Y$.
\end{proof}

\begin{theorem} \label{main_result}
The following are equivalent for a pseudocompact space $X$:

(i) $X$ is metrizable;

(ii) $X^2\setminus\diag$ is $(\omega_1,\omega)$-paracompact; and

(iii) $X^2\setminus\diag$ is $P$-paracompact for some directed set $P$ with calibre $(\omega_1,\omega)$.
\end{theorem}

\begin{proof}
We have (i) $\Rightarrow$ (iii) since `paracompact' is equivalent to `$1$-paracompact', and  (iii) $\Rightarrow$ (ii) by Lemma~\ref{Pcalom1om_imp_relom1om}, so we just need to show (ii) $\Rightarrow$ (i).

First, pick any point $x\in X$ and note that, since $X\setminus\{x\}$ is homeomorphic to a closed subspace of $X^2\setminus\diag$, then $X\setminus\{x\}$ is $P$-paracompact, and so $X$ is compact by Lemmas~\ref{card} and~\ref{pseudo}.

Now, $X^2\setminus\diag$ is locally compact and $(\omega_1,\omega)$-paracompact, so we can find an open cover $\mathcal{U}$ of $X^2\setminus\diag$ which is  $(\omega_1,\omega)$-locally finite and such that $\cl{U}^{X^2\setminus\diag}$ is compact for each $U \in \mathcal{U}$.  For each $n\in\N$, define a relation $\sim_n$ on $\mathcal{U}$ by $U \sim_n V$ if and only if there are $U_0,U_1,\ldots,U_n \in \mathcal{U}$ such that $U_0 = U$, $U_n = V$, and $U_i \cap U_{i-1} \neq \emptyset$ for each $i=1,\ldots,n$.  Then define an equivalence relation $\sim$ on $\mathcal{U}$ by $U\sim V$ if and only if $U \sim_n V$ for some $n\in\N$.

Now let $\{[U_\alpha] : \alpha < \kappa\}$ be a one-to-one enumeration of the $\sim$-equivalence classes, and let $S_\alpha = \bigcup [U_\alpha]$.  Then $\{S_\alpha : \alpha < \kappa\}$ is a partition of $X^2\setminus\diag$ consisting of open sets.  Each $S_\alpha$ is thus also closed in $X^2\setminus\diag$, and it follows that $S_\alpha = \bigcup \{\cl{U}^{X^2\setminus\diag} : U \in [U_\alpha]\}$.  Hence, we can show that each $S_\alpha$ is $\sigma$-compact by verifying that $[U_\alpha]$ is countable.

For each $U \in \mathcal{U}$ and $n\in\N$, let $[U]_n = \{V \in \mathcal{U} : U \sim_n V\}$.  We first prove that $[U]_1$ is countable for each $U \in\mathcal{U}$.  Suppose, instead, that $[U]_1$ is uncountable.  Since $\mathcal{U}$ has relative calibre $(\omega_1,\omega)$ in $LF(\mathcal{U})$, then there is an infinite subset $\mathcal{V} \subseteq [U]_1$ which is locally finite.  But since $\cl{U}^{X^2\setminus\diag}$ is compact and $\mathcal{V}$ is locally finite, then there should be only finitely many members of $\mathcal{V}$ intersecting $U$, which is a contradiction since every member of $[U]_1$ intersects $U$.  Now since $[U]_{n+1} = \bigcup \{[V]_1 : V \in [U]_n\}$, then by induction, we see that each $[U]_n$ is countable, so $[U] = \bigcup \{[U]_n : n\in\N\}$ is also countable.

Suppose $X$ is not metrizable.  Since each $S_\alpha$ is $\sigma$-compact, then by Lemma~\ref{GrLemma}, we can find a subspace $Y$ of $X$ and an uncountable cardinal $\lambda$ such that $Y$ is homeomorphic to $A(\lambda)$.  Since $Y$ is compact, then $Y^2\setminus\diag$ is a closed subspace of $X^2\setminus\diag$, so $Y^2\setminus\diag$ is also $(\omega_1,\omega)$-paracompact.  But this is a contradiction, according to Lemma~\ref{supersequence}.
\end{proof}

The next result shows that in statement (iii) of Theorem~\ref{main_result} we cannot replace the class of all directed sets with calibre $(\omega_1,\omega)$ with a larger class of directed sets.

\begin{theorem}
Suppose $P$ is a directed set satisfying: every compact space $X$ such that $X^2 \setminus \diag$ is $P$-paracompact must be metrizable. Then $P$ must be calibre $(\omega_1,\omega)$.
\end{theorem}

\begin{proof}
We show the contrapositive. Suppose $P$ is a directed set which is not calibre $(\omega_1,\omega)$. Take any compact space $X$ which has weight precisely $\omega_1$ (for example, $A(\omega_1)$). By Lemma~\ref{general_P}, $X^2 \setminus \diag$ has a $P$-(locally) finite base, and so is $P$-paracompact, but $X$ is not metrizable.
\end{proof}

\subsection{Diversity of $\K(M)$-metrizable Spaces}

The next theorem says that there exist $\K(M)$-metrizable spaces for every separable metrizable $M$. Further, there is a maximal `antichain' of separable metrizable spaces with corresponding topological spaces which are $\K(M)$-metrizable for one, and only one, member $M$ of the antichain.

\begin{theorem}\label{main_conv}
For each separable metrizable space $A$, there is a hereditarily paracompact $\K(A)$-metrizable space $M_A$ such that:
 if $A'$ is any non-compact separable metrizable space and $M_A$ is $\K(A')$-metrizable, then $\K(A') \ge_T (A,\K(A))$.

Hence there is a $2^\ctm$-sized family $\mathcal{A}$ of separable metrizable spaces such that:

(i) if $A$ is in $\mathcal{A}$ then $M_A$ is $\K(A)$-metrizable, but

(ii) if $A'$ is another member of $\mathcal{A}$, then $M_A$ is not $\K(A')$-metrizable.
\end{theorem}

\begin{proof}
Fix a separable metrizable space $A$. Without loss of generality, we suppose $A$ is a subspace of the Hilbert cube, $I^\N$, and set $M_A = M(I^\N,A)$. 

By Lemma~\ref{Mpmet}, $M_A$ is $\K(A)$-metrizable. Let $A'$ be any non-compact separable metrizable space and suppose $M_A$ is $\K(A')$-metrizable. By Lemma~\ref{Mpmet} $(\K(A') \times \N) \ge_T (A, CL(I^\N) \cap \mathbb{P}(A))$.  Since $A'$ is not compact,  $\K(A')\times \N =_T \K(A')$ and $CL(I^\N) \cap \mathbb{P}(A) = \K(A)$, so we have $\K(A') \ge_T (A,\K(A))$.

Take $\mathcal{A}$ to be the $2^\ctm$-sized `antichain' of Theorem~\ref{bigantichain}. No member of $\mathcal{A}$ is compact. The first part of this theorem and the properties of the antichain then immediately yield (i) and (ii).
\end{proof}

\subsection{Characterizing $\K(M)$-paracompact and $\K(M)$-metrizable}

Here we aim to give characterizations of when a space is $\K(M)$-paracompact or $\K(M)$-metrizable, for some separable metrizable $M$, in terms of properties not referring to a separable metrizable space. This is completely successful for $\K(M)$-metrizability but only partially so for $\K(M)$-paracompactness. The characterizations provide insight into the structure of $\K(M)$-paracompact and $\K(M)$-metrizable spaces that are key to all subsequent results. We give examples showing that the results do not hold if the given additional hypotheses are dropped, nor if  $\K(M)$-paracompact/$\K(M)$-metrizable is weakened to $P$-paracompact/$P$-metrizable where $P$ has calibre $(\omega_1,\omega)$. We also give examples distinguishing all the relevant properties ($\K(M)$-metrizable, $P$-metrizable for $P$ with calibre $(\omega_1,\omega)$, $(\omega_1,\omega)$-metrizable, etc).

\begin{theorem} \label{P_metric}
Let $X$ be a space.

(i) $X$ is $\K(M)$-metrizable for some separable metrizable $M$ if and only if $X$ has a weakly $\sigma$-locally finite base.

(ii) If $X$ is $\K(M)$-metrizable for some separable metrizable $M$, then $X$ has a $\sigma$-relatively locally finite base.

(iii) If $X$ is $\K(M)$-perfectly normal for some separable metrizable $M$ and has a $\sigma$-relatively locally finite base, then it is $\K(M)$-metrizable.

(iv) If $X$ has a $\sigma$-disjoint base, then it has a $\sigma$-relatively locally finite base. If $X$ has a $\sigma$-relatively locally finite base, then $X$ has a $\sigma$-point finite base.

(v) If $X$ is $\K(M)$-metrizable for some separable metrizable $M$, then it is $P$-metrizable where $P$ has calibre $(\omega_1,\omega)$.

(vi) If $X$ is $P$-metrizable where $P$ has calibre $(\omega_1,\omega)$, then $X$ is $(\omega_1,\omega)$-metrizable.
\end{theorem}

\begin{proof}
Proposition~\ref{wslf_kmlf} shows that any weakly $\sigma$-locally finite base is $\K(M)$-locally finite for some separable metrizable space $M$, which gives one direction of (i). Conversely, if $X$ is $\K(M)$-metrizable for some separable metrizable space $M$, then $X$ is first countable by Lemma~\ref{om1omMetis1o}. So by Proposition~\ref{wslf_kmlf}, we can see that $X$ has a weakly $\sigma$-locally finite base, which completes the proof of (i). 

Statement (ii) follows from (i) and Lemma~\ref{wslf_to_srlf}, and (iii) is just a special case of Lemma~\ref{srlf_ppn_to_pm}.

Since any $\sigma$-disjoint open family is $\sigma$-relatively locally finite, and any $\sigma$-relatively locally finite family is $\sigma$-point finite, then we can see that (iv) is true.

Claims (v) and (vi) follow from Lemmas~\ref{when_calom1om} and~\ref{Pcalom1om_imp_relom1om}, respectively.
\end{proof}

We can summarize these results as follows. Examples on their own but next to an arrow representing an implication demonstrate that a converse fails. A `$+$ property' indicates an additional hypothesis, and the adjacent example shows that the additional hypothesis is necessary. `Perfectly normal' has been abbreviated `PN'. 
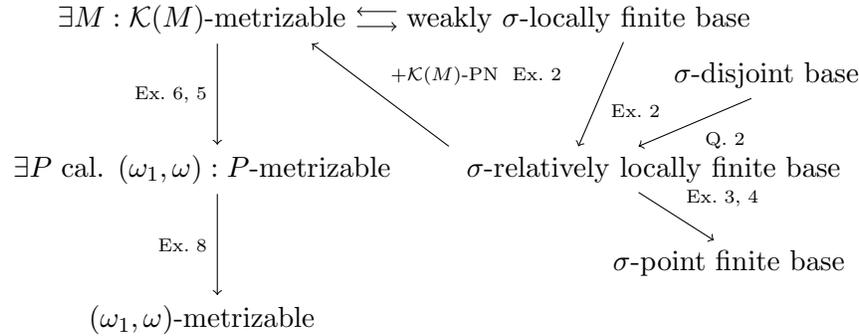
\begin{center}
\begin{tikzpicture}[auto,prop/.style={rectangle}]

\node[prop] (wsp) at (5,4) {weakly $\sigma$-locally finite base};

\node[prop] (kMp) at (0,4) {$\exists M : \K(M)$-metrizable};

\node[prop] (srp) at (6,2) {$\sigma$-relatively locally finite base};

\node[prop] (scr) at (7.5,3.25) {$\sigma$-disjoint base};

\node[prop] (sm) at (7,0.75) {$\sigma$-point finite base};

\node[prop] (pcal) at (0,2) {$\exists P$ cal. $(\omega_1,\omega) : P$-metrizable};

\node[prop] (om1omp) at (0,0) {$(\omega_1,\omega)$-metrizable};

\draw[<-] ($(wsp.west)+(0,-0.1)$) to ($(kMp.east)+(0,-0.1)$);
\draw[<-] ($(kMp.east)+(0,0.1)$) to  ($(wsp.west)+(0,0.1)$);

\draw[->] ($(kMp.south)+(0.2,0)$) to node [swap]  {\tiny{Ex.~\ref{PpctnotkM}, \ref{bigM}}} ($(pcal.north)+(0.2,0)$);

\draw[->] ($(pcal.south)+(0.2,0)$) to node [swap] {\tiny{Ex.~\ref{om1ompcptnotPpcpt}}} ($(om1omp.north)+(0.2,0)$);

\draw[->] ($(srp.north west)+(-0.1,-0.0)$) to node [swap] {\tiny{$+\K(M)$-PN \ Ex.~\ref{srpnotkmp}}} ($(kMp.south east)+(-0.6,0)$);

\draw[->] ($(scr.south)+(-0.2,0)$) to node  {\tiny{Q.~\ref{srpnotscr}}} ($(srp.north)+(-0.2,0)$);

\draw[->] ($(srp.south)+(-0.2,0)$) to node  {\tiny{Ex.~\ref{spfnotsrp}, \ref{spfnotkmp}}} ($(sm.north)+(-0.2,0)$);

\draw[->] ($(wsp.south)+(0.6,0)$) to node  {\tiny{Ex.~\ref{srpnotkmp}}} ($(srp.north)+(-1,0)$); 
\end{tikzpicture}

\end{center}

There is a clear logical difference between saying that a space is `metrizable' or is `paracompact'. Metrizability asserts the existence of a certain object ($\sigma$-locally finite base), while paracompactness says that for every object of one type (open cover) there is a certain object of another type (locally finite open refinement). This logical difference means that there is a unique `$\K(M)$-variant' of metrizability ($\K(M)$-metrizable, for some $M$) but two `$\K(M)$-variants' of paracompactness, depending on whether the $M$ used to organize an open refinement is chosen in advance ($\K(M)$-paracompact) or with the refinement ($\K(\M)$-paracompact). This results in a more complex range of implications and examples.

\begin{theorem}\label{P_para}
Let $X$ be a space.

(i) If $X$ is $\K(M)$-paracompact for some separable metrizable $M$, then $X$ is $\K(\M)$-paracompact.

(ii) If $X$ is $\K(\M)$-paracompact, then $X$ is $P$-paracompact for some directed set $P$ with calibre $(\omega_1,\omega)$.

(iii) If $X$ is $P$-paracompact for some directed set $P$ with calibre $(\omega_1,\omega)$, then $X$ is $(\omega_1,\omega)$-paracompact.

(iv) If $X$ is weakly $\sigma$-paracompact, then it is $\K(\M)$-paracompact and $\sigma$-relatively paracompact.

(v) If $X$ is first countable and $\K(\M)$-paracompact, then $X$ is weakly $\sigma$-paracompact (and hence, $\sigma$-relatively paracompact).

(vi) If $X$ is $\K(M)$-perfectly normal for some separable metrizable $M$ (respectively, $\K(\M)$-perfectly normal) and is $\sigma$-relatively paracompact, then it is $\K(M)$-paracompact (respectively, $\K(\M)$-paracompact).

(vii) If $X$ is screenable, then it is $\sigma$-relatively paracompact. If $X$ is $\sigma$-relatively paracompact, then $X$ is $\sigma$-metacompact.
\end{theorem}

\begin{proof} Claim (i) is immediate from the definitions. Claim (iii) follows from Lemma~\ref{Pcalom1om_imp_relom1om}.

For (ii) suppose $X$ is $\K(\M)$ paracompact. Let $P= \Sigma \{\K(M) : M \subseteq I^\N\}$.  Then $P$ has calibre $(\omega_1,\omega)$ by Theorem~\ref{Sigma}. Take any open cover $\mathcal{U}$ of $X$. By hypothesis there is a separable metrizable $M$ such that $\mathcal{U}$ has a $\K(M)$-locally finite open refinement $\mathcal{V}$. Without loss of generality, we can suppose $M$ is a subspace of the Hilbert cube $I^\N$. Taking the relevant projection, clearly $P \ge_T \K(M)$. Hence $\mathcal{V}$ is $P$-locally finite, as required.

For (iv) suppose $\mathcal{W}$ is a weakly $\sigma$-locally finite open cover of $X$. By Proposition~\ref{wslf_kmlf}, $\mathcal{W}$ is also $\K(M)$-locally finite for some separable metrizable $M$. By Lemma~\ref{wslf_to_srlf}, $\mathcal{W}$ has a $\sigma$-relatively locally finite open refinement $\mathcal{V}$.

Similarly, (v) follows immediately from Proposition~\ref{wslf_kmlf} and Lemma~\ref{wslf_to_srlf}, while (vi) is just a special case of Lemma~\ref{srp_ppn_to_pp}.

Since any $\sigma$-disjoint open family is $\sigma$-relatively locally finite, and any $\sigma$-relatively locally finite open family is also $\sigma$-point finite, we see that (vii) is true.
\end{proof}

Again, we diagrammatically summarize these results and indicate relevant examples.

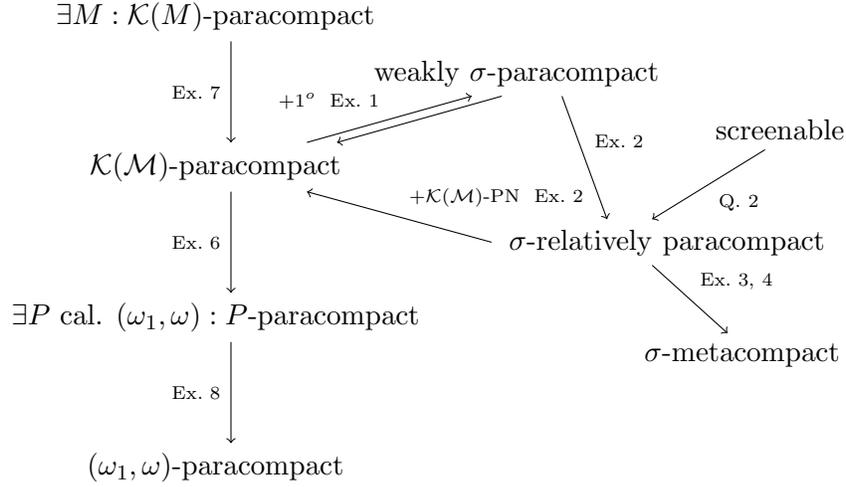
\begin{center}
\begin{tikzpicture}[auto,prop/.style={rectangle}]

\node[prop] (wsp) at (4,5.25) {weakly $\sigma$-paracompact};

\node[prop] (kMp) at (0,4) {$\K(\M)$-paracompact};

\node[prop] (srp) at (6,3) {$\sigma$-relatively paracompact};

\node[prop] (kmp) at (0,6) {$\exists  M  :  \K(M)$-paracompact};

\node[prop] (scr) at (7.5,4.5) {screenable};

\node[prop] (sm) at (7,1.5) {$\sigma$-metacompact};

\node[prop] (pcal) at (0,2) {$\exists P$ cal. $(\omega_1,\omega) : P$-paracompact};

\node[prop] (om1omp) at (0,0) {$(\omega_1,\omega)$-paracompact};

\draw[->] ($(wsp.south)+(-0.2,0)$) to ($(kMp.north east)+(-0.2,0)$);
\draw[->] ($(kMp.north east)+(-0.6,0)$) to node  {\tiny{$+1^o$ \ Ex.~\ref{kmnotsrp}}} ($(wsp.south)+(-0.6,0)$);

\draw[->] ($(kmp.south)+(0.2,0)$) to node [swap] {\tiny{Ex.~\ref{kMnotkm}}} ($(kMp.north)+(0.2,0)$);

\draw[->] ($(kMp.south)+(0.2,0)$) to node [swap]  {\tiny{Ex.~\ref{PpctnotkM}}} ($(pcal.north)+(0.2,0)$);

\draw[->] ($(pcal.south)+(0.2,0)$) to node [swap] {\tiny{Ex.~\ref{om1ompcptnotPpcpt}}} ($(om1omp.north)+(0.2,0)$);

\draw[->] ($(srp.west)+(-0.1,-0.0)$) to node [swap] {\tiny{$+\K(\M)$-PN \ Ex.~\ref{srpnotkmp}}} ($(kMp.south east)+(-0.6,0)$);

\draw[->] ($(scr.south)+(-0.2,0)$) to node  {\tiny{Q.~\ref{srpnotscr}}} ($(srp.north)+(-0.2,0)$);

\draw[->] ($(srp.south)+(-0.2,0)$) to node  {\tiny{Ex.~\ref{spfnotsrp}, \ref{spfnotkmp}}} ($(sm.north)+(-0.2,0)$);

\draw[->] ($(wsp.south)+(0.6,0)$) to node  {\tiny{Ex.~\ref{srpnotkmp}}} ($(srp.north)+(-0.8,0)$); 
\end{tikzpicture}
\end{center}

\begin{exam}[$\neg$CH]\label{kmnotsrp}
There is a $\K(M)$-paracompact space (hence $\K(\M)$-paracompact) which is not $\sigma$-relatively paracompact (hence not weakly $\sigma$-paracompact).
\end{exam}

\begin{proof}
Consider the space $X(I_\omega)$ where $I = [0,1]$.  Note that $w(I)=\omega$, so if we assume $\neg$CH, then $w(I)\cdot \omega_1<\mathfrak{c}$. Then by (ii) in Lemma~\ref{co_countable}, we have that $X(I_\omega)$ is not $\sigma$-relatively paracompact.

However, we also know from Lemma~\ref{co_countable} that $X(I_\omega)$ is $[I]^{\le\omega}$-paracompact, and we know from Lemma~\ref{bernstein} that $\K(M) \ge_T [\mathfrak{c}]^{\le \omega} = [I]^{\le \omega}$, where $M=B$ is a Bernstein set.  So $X(I_\omega)$ is $\K(M)$-paracompact. 
\end{proof}

Curiously, under CH the same example is paracompact.

\begin{exam}\label{srpnotkmp}
There is a  a Moore space which  has a $\sigma$-disjoint base (and hence has a $\sigma$-relative locally finite base) which is not $(\omega_1,\omega)$-paracompact, and so not $\K(\M)$-paracompact.
\end{exam}

\begin{proof}
Let $X=S(I)$. This space is well known to be a Moore space with a $\sigma$-disjoint base, and so, as observed above, it has a $\sigma$-relatively locally finite base.

The space $X$ is not $(\omega_1,\omega)$-paracompact (see Lemma~\ref{sy_om1om_pcpt}), and so not $\K(M)$-paracompact for any separable metrizable $M$
\end{proof}

\begin{exam}\label{spfnotsrp}
There is a Moore space with a $\sigma$-point finite base (hence, $\sigma$-metacompact) which is not $\sigma$-relatively paracompact.
\end{exam}

\begin{proof}
Take $X=X(\R)$. Since $w(\R)=\omega<|\R|$, we have that $X(\R)$ is not $\sigma$-relatively paracompact by Lemma~\ref{not_rp}. By Lemma~\ref{X_basic_facts}, we have that $X(\R)$ is a Moore space. 
\end{proof}

\begin{exam}[$\exists$ a $Q$-set]\label{spfnotkmp}\label{spfnotkm}
There is a  Moore space  which is perfectly normal, has a $\sigma$-point finite base, but is not $\sigma$-relatively paracompact.
\end{exam}

\begin{proof}
Let $Q$ be a $Q$-set. Then $X(Q)$ is well-known to be a non-metrizable, perfectly normal Moore space with a $\sigma$-point finite base. The rest follows as in the preceding example.
\end{proof}

\begin{exam}\label{bigM}
There is a hereditarily paracompact, first countable space which is $P$-metrizable for a directed set $P$ with calibre $(\omega_1,\omega)$, but is not $\K(M)$-metrizable for any separable metrizable space $M$.
\end{exam}

\begin{proof}
Take a family $\{A_\alpha \subseteq I^\N : \alpha < \ctm^+\}$ of distinct subsets of the Hilbert cube, $I^\N$, and let $P = \Sigma \{\K(A_\alpha) : \alpha<\ctm^+\}$, which has calibre $(\omega_1,\omega)$ by Theorem~\ref{Sigma}.  Then Lemma~\ref{YAforSigmaKA} provides a metrizable space $Y$ with a subspace $A$ such that $P \times \N \ge_T P \ge_T (A,\mathop{CL}(Y) \cap \mathbb{P}(A))$.  

Let $X = M(Y,A)$. Then $X$ is first countable and hereditarily paracompact, and $X$ is $P$-metrizable by Lemma~\ref{Mpmet}.

Suppose $X$ is $\K(M)$-metrizable for some separable metrizable space $M$, and let $M' = M \times \N$.  Then Lemma~\ref{Mpmet} implies that $\K(M') =_T \K(M)\times\N \ge_T (A,\mathop{CL}(Y) \cap \mathbb{P}(A))$.  Thus, $\K(M') \ge_T (A_\alpha,\K(A_\alpha))$ for each $\alpha < \ctm^+$ by part (ii) of Lemma~\ref{YAforSigmaKA}, but that contradicts (ii) in Theorem~\ref{km_bounded}.
\end{proof}

\begin{exam}\label{PpctnotkM} There is a Moore space with a $\sigma$-disjoint base which is $P$-metrizable for a directed set $P$ with calibre $(\omega_1,\omega)$, but is not $\K(\M)$-paracompact.
\end{exam}

\begin{proof}
Let $Y$, $A$, and $P$ be as in the proof of Example~\ref{bigM}, and define $B = Y \setminus A$.  By Lemma~\ref{lf_sum_sm}, there is a $Q = \Sigma \{M_\alpha : \alpha < \kappa\}$, where each $M_\alpha$ is separable metrizable, such that $Q \ge_T (B, CL(Y) \cap \mathbb{P}(B))$.  Then $P' = P \times Q$ is also a $\Sigma$-product of $\K(M)$'s and so has calibre $(\omega_1,\omega)$ by Theorem~\ref{Sigma}.

Arguing as in Example~\ref{bigM}, but with Lemma~\ref{Dpmet} replacing Lemma~\ref{Mpmet}, we see that $X = D(Y;A,B)$ is $P'$-paracompact but not $\K(M)$-paracompact for any separable metrizable space $M$.  Since $Y$ is first countable, then $X$ is $P'$-metrizable by Lemma~\ref{D_basics}. Since $D(Y;A,B)$ is a Moore space, it follows from Lemma~\ref{ppms_to_pnm} that it cannot be $\K(\M)$-paracompact. 
\end{proof}

\begin{exam}\label{kMnotkm} There is a space which is 
$\K(\M)$-paracompact but not $\K(M)$-paracompact for any separable metrizable $M$.
\end{exam}

\begin{proof}
Let $\{A_\alpha : \alpha < \ctm^+\}$ be a family of distinct subsets of $I^\N$. For each $\alpha < \ctm^+$, set $Y_\alpha = I^\N$, $B_\alpha = Y_\alpha \setminus A_\alpha$, and $X_\alpha = D(Y_\alpha;A_\alpha,B_\alpha)$. Then define $X = \bigoplus_\alpha X_\alpha$, and let $X^*$ be $X$ with one additional point, $\ast$, where basic neighborhoods of $\ast$ have the form $U_C = \{\ast\} \cup \bigoplus \{ X_\alpha : \alpha \in \ctm^+ \setminus C\}$ for any countable subset $C$ of $\ctm^+$.

Fix a separable metrizable space $M$. We check that $X^*$ is not $\K(M)$-paracompact.  By Theorem~\ref{km_bounded}, we have $\K(M \times \N) \not\ge_T (A_\alpha,\K(A_\alpha))$ for some $\alpha < \ctm^+$, and if  $X^*$ were $\K(M)$-paracompact, then the closed subspace $X_\alpha$ would also be $\K(M)$-paracompact. But since $Y_\alpha$ is metrizable, then by Lemma~\ref{Dpmet}, we would have $\K(M \times \N) =_T \K(M) \times \N \ge_T (A_\alpha,CL(Y) \cap \mathbb{P}(A))$.  However, $CL(Y) \cap \mathbb{P}(A) = \K(A_\alpha)$ since $Y_\alpha$ is compact, which gives a contradiction.

Now we show $X^*$ is $\K(\mathcal{M})$-paracompact. Let $\mathcal{U}$ be any open cover of $X^*$, and pick a $U_*$ in $\mathcal{U}$ containing $\ast$.  Then there is a countable subset $C$ of $\ctm^+$ such that $U_*$ contains $X_\alpha$ for each $\alpha \in \ctm^+ \setminus C$. By Lemma~\ref{Dpmet}, each $X_\alpha$ is $\K(M_\alpha)$-paracompact where $M_\alpha = A_\alpha \times B_\alpha$.  By Theorem~\ref{km_bounded}, there is a separable metrizable $M$ such that $M \ge_T M_\alpha$ for each $\alpha$ in $C$.  Thus, for each $\alpha$ in $C$, $X_\alpha$ is $\K(M)$-paracompact, so we can find a $\K(M)$-locally finite open refinement $\mathcal{V}_\alpha = \bigcup \{\mathcal{V}_{\alpha,K} : K\in\K(M)\}$ of $\mathcal{U}_\alpha = \{U \cap X_\alpha : U \in \mathcal{U}\}$.  For each $K\in\K(M)$, define $\mathcal{V}_K = \{U_*\} \cup \bigcup \{\mathcal{V}_{\alpha,K} : \alpha \in C\}$.  Then $\mathcal{V} = \bigcup \{\mathcal{V}_K : K\in\K(M)\}$ is a $\K(M)$-locally finite open refinement of $\mathcal{U}$ since each $X_\alpha$ is open in $X^*$.
\end{proof}

\begin{ques}\label{srpnotscr}
Is there a $\sigma$-relatively paracompact space which is not screenable?
Is there a $\K(M)$-metrizable space without a $\sigma$-disjoint base? One which is a Moore space?
\end{ques}

\begin{exam}\label{om1ompcptnotPpcpt}
There is a Moore space with a $\sigma$-disjoint base which is $(\omega_1,\omega)$-metrizable but not $P$-paracompact for any directed set $P$ with calibre $(\omega_1,\omega)$.
\end{exam}
\begin{proof}
Consider Heath's original split V space, $H=H(\R)$. This is a Moore space with a $\sigma$-disjoint base. By Lemma~\ref{Hom1ommet},  $H$ is $(\omega_1,\omega)$-metrizable. Lemma~\ref{HnotPpcpt} implies $H$ is not $P$-paracompact for any directed set $P$ with calibre $(\omega_1,\omega)$.
\end{proof}

We know that `first countable plus $\K(\M)$-paracompact' implies `$\sigma$-relatively paracompact' and `$\K(M)$-metrizable' implies `$\sigma$-relatively locally finite base'. The last pair of examples show that `$\K(\M)$-paracompact' cannot be replaced above by `$(\omega_1,\omega)$-paracompact' or `$P$-paracompact, where $P$ has calibre $(\omega_1,\omega)$'; nor can `$\K(M)$-metrizable' be similarly weakened. 

\begin{exam}\label{Mooreom1omnotsrp}
There is a Moore space (hence, first countable) which is $(\omega_1,\omega)$-paracompact but not  $\sigma$-relatively paracompact, and so it is not $\K(\M)$-paracompact and not weakly $\sigma$-paracompact.
\end{exam}

\begin{proof}
Take $X=X(B)$, where $B$ is a Bernstein set. We can see that $B$ is RCCC since every compact subset of $B$ is countable. By Lemma~\ref{RCCC}, $X(B)$ is $(\omega_1, \omega)$-paracompact. Since $|B|>\omega_0$ and $w(B)=\omega_0$, we see that $X(B)$ is not $\sigma$-relatively paracompact by Lemma~\ref{not_rp}.
\end{proof}

\begin{exam}\label{Ppcpnotsrp}
There is a $P$-paracompact space, where $P$ has calibre $(\omega_1,\omega)$, which is not $\sigma$-relatively paracompact, and also not $\K(\M)$-paracompact.
\end{exam}

\begin{proof}
Let $Y=I^\mathfrak{c}$, and consider $X=X(Y_\omega)$. By (i) in Lemma~\ref{co_countable}, $X=X(Y_\omega)$ is $[Y]^{\leq \omega}$-paracompact. Since the directed set $[Y]^{\leq \omega}$ has calibre $(\omega_1,\omega)$, we have that $X(Y_\omega)$ is $(\omega_1, \omega)$-paracompact.

Notice that $w(Y)= \mathfrak{c}$. Therefore $w(Y)\cdot \omega_1<2^{\mathfrak{c}}$. By (ii) in Lemma~\ref{co_countable}, we have that $X(Y_\omega)$ is not $\mathfrak{c}$-relatively paracompact, and hence, it is not $\sigma$-relatively paracompact. Again, $X(Y_\omega)$ is not $\mathfrak{c}$-relatively paracompact, and so not $\ctm$-paracompact.  Since for every $M$ in $\M$ we have $\cof{\K(M)} \le  \ctm$,  from Lemma~\ref{kMpcptIscpct} we see that $X(Y_\omega)$  is not $\K(\M)$-paracompact.
\end{proof}

There remains a gap here, the first example is first countable but not $\K(M)$-paracompact, while the second is not first countable.

\begin{ques} Is there a space which is first countable, $P$-paracompact for some $P$ with calibre $(\omega_1,\omega)$, but not $\sigma$-relatively paracompact, or even not $\sigma$-metacompact?
Is there an example which is $P$-metrizable?
\end{ques}

\subsection{With Normal, Countably Paracompact or CCC}

Recall that a space $X$ is \emph{countably paracompact} if and only if every increasing countable open cover $\{U_n : n \in \N\}$ of $X$ is shrinkable, i.e., it has an open refinement $\{V_n : n \in \N\}$ such that $\cl{V_n} \subseteq U_n$; and $X$ is normal and countably paracompact if and only if every countable (not necessarily increasing) open cover of $X$ is shrinkable.

\begin{theorem}\label{char_pcpt}
Let $X$ be a  first countable space.
Then $X$ is paracompact if and only if it is $\K(\M)$-paracompact, normal and countably paracompact.
\end{theorem}

\begin{proof}
One direction is immediate.  For the other direction, take any open cover $\mathcal{U}$ of $X$. Since first countable $\K(\M)$-paracompact spaces are $\sigma$-relatively paracompact (Theorem~\ref{P_para}(v)), we can get an open refinement $\mathcal{W}$ of $\mathcal{U}$ which can be written $\mathcal{W}=\bigcup \{\mathcal{W}_n : n \in \N\}$ where $\mathcal{W}_n$ is locally finite in $X_n=\bigcup \mathcal{W}_n$.

Shrink $\{X_n : n \in \N\}$ to get open an open cover $\{Y_n : n \in \N\}$ such that $\cl{Y_n} \subseteq X_n$. Let $\mathcal{T}_n = \{ W \cap Y_n : W \in \mathcal{W}_n\}$. Then $\mathcal{T}_n$ is locally finite, and so $\mathcal{U}$ has a $\sigma$-locally finite open refinement. Hence $X$ is paracompact.
\end{proof}

Since $\K(M)$-metrizable spaces are first countable (Lemma~\ref{om1omMetis1o}), we deduce:
\begin{theorem}
Every $\K(M)$-metrizable space which is normal and countably paracompact is paracompact.
\end{theorem}

Then from the above and Lemma~\ref{ppms_to_pnm} it follows:
\begin{theorem}\label{NM}
Every $\K(\M)$-paracompact, normal Moore space is metrizable.
\end{theorem}

These theorems raise some interesting questions connected to Dowker spaces and the Normal Moore Space Conjecture.

\begin{ques}\label{km_dowk_q}
Is there a (first countable) $\K(\M)$-paracompact normal space which is not countably paracompact?  Is there a $\K(M)$-metrizable example?
\end{ques}
Since there is no known example of a Dowker space (normal, not countably paracompact) which is first countable and $\sigma$-relatively paracompact, we seem a long way from answering this question. Rudin \cite{Ru} (under $\diamondsuit^+$) and Balogh \cite{Ba} (in ZFC) do have examples of screenable (hence, $\sigma$-relatively paracompact) Dowker spaces. It is not clear to the authors if either of these examples is $\K(M)$-paracompact.

The other direction to go from the above theorems is to ask if normality can be dropped (and replaced with countable paracompactness). Here we have a consistent counter-example.

\begin{exam}[Consistently]

There is a  $\K(M)$-metrizable, countably paracompact Moore space which is not normal.
\end{exam}

\begin{proof} Knight \cite{Knight89} has shown it is consistent that there is a $\Delta$-set $A$ which is not a $Q$-set. Fix a subset $A_1$ of $A$ which is not a $G_{\delta}$, and let $A_2=A \setminus A_1$. Then $X=D(A;A_1,A_2)$ is a Moore space, $\K(M)$-metrizable for $M = A_1\times A_2$ (see Lemmas~\ref{D_basics} and \ref{Dpmet}), and  countably paracompact (Lemma~\ref{Dctblypcpt}).

As $A_2$ is not an $F_\sigma$ subset of $X$, then Lemma~\ref{Dpmet} also tells us that $X$ is not metrizable and so not normal (by Theorem~\ref{NM}).
\end{proof}

For any directed set $P$ with calibre $(\omega_1,\omega)$, we know (Lemma~\ref{densesubset}) that `separable plus $P$-paracompact implies Lindel\"{o}f' and `separable plus $P$-metrizable implies metrizable'.
It is natural to ask when  `separable' can be relaxed to `ccc' (every pairwise disjoint family of open sets is countable)?

The next lemma is well-known.
\begin{lemma}
Every locally finite open cover $\mathcal{W}$ of a ccc space $Y$ contains a countable subcollection whose closures cover $Y$.
\end{lemma}

Now we can give a positive answer to our question in the case when $P$ is a $\K(M)$.
\begin{theorem}\label{kmccc}
Let $X$ be a ccc space.

(i) If $X$ is first countable and $\K(\M)$-paracompact then $X$ is Lindel\"{o}f.

(ii) If $X$ is $\K(M)$-metrizable then $X$ is metrizable.
\end{theorem}

\begin{proof}
We prove (i) first. So suppose $X$ is first countable and $\K(\M)$-paracompact. Then $X$ is $\sigma$-relatively paracompact. Take any open cover $\mathcal{U}$. It has an open refinement $\mathcal{V}=\bigcup_n \mathcal{V}_n$ where each $\mathcal{V}_n$ is relatively locally finite, and (using regularity)  we can additionally assume that the closure of each $V$ in $\mathcal{V}$ is contained in some member of $\mathcal{U}$.

Fix $n$. Apply the preceding lemma to the ccc space $Y_n=\bigcup \mathcal{V}_n$ and the locally finite cover $\mathcal{V}_n$ to get a countable subcollection of $\mathcal{V}_n$ whose closures cover $Y_n$. Recalling that the closure of each $V$ in $\mathcal{V}$ is contained in some member of $\mathcal{U}$, we obtain a countable subcollection of $\mathcal{U}$ covering $\bigcup \mathcal{V}_n$.  Taking the union over all $n$ of these countable subcollections yields a countable subcover of $\mathcal{U}$.

Now we establish (ii). Suppose $X$ is $\K(M)$-metrizable. Then every subspace is $\K(M)$-paracompact. In particular, every open subspace is $\K(M)$-paracompact and ccc, and hence Lindel\"{o}f. Thus $X$ is hereditarily Lindel\"{o}f, so hereditarily ccc, and hence (see \cite{Ga} for example)  any point-finite family of open sets in $X$ is countable. But as $X$ is $\K(M)$-metrizable, it has a $\sigma$-point finite base, which we now see must be countable. Thus $X$ is indeed (separable and) metrizable.
\end{proof}

It is unknown to the authors if the restriction to first countable spaces is necessary in part (i) above.
\begin{ques}
Is there a ccc $\K(M)$-paracompact space which is not Lindelof?
\end{ques}

Nor do we know what happens for $P$ not of the form $\K(M)$. The machines developed above yield spaces that are far from being ccc.
\begin{ques}
Is there a first countable, ccc, $P$-paracompact space, where $P$ has calibre $(\omega_1,\omega)$, which is not Lindelof? Is there a  ccc, $P$-metrizable space, where $P$ has calibre $(\omega_1,\omega)$, which is not metrizable?
\end{ques}

\end{document}